\documentclass[12pt]{amsart}
\usepackage{snowshovelingalg} 
\usepackage[ruled,vlined, linesnumbered]{algorithm2e}
\usepackage{graphicx, subcaption, mdframed, float}
\usepackage{pgfplots}
\usepgfplotslibrary{colormaps}
\usepackage{alphalph}
\newcommand{\omegaParameters}{\Lambda}

\renewcommand{\i}{^{-1}}

\SetKwInOut{Parameter}{Parameters}

\let\oldnl\nl
\newcommand{\nonl}{\renewcommand{\nl}{\let\nl\oldnl}}

\makeatletter
\@namedef{subjclassname@2020}{%
  \textup{2020} Mathematics Subject Classification}
\makeatother

\begin{document}

\title{$\mathcal{W}_\infty$-transport with discrete target as a combinatorial matching problem}

\author{Mohit Bansil}
\address{Department of Mathematics, Michigan State University}
\email{bansilmo@msu.edu}

\author{Jun Kitagawa}
\address{Department of Mathematics, Michigan State University}
\email{kitagawa@math.msu.edu}
\thanks{JK's research was supported in part by National Science Foundation grants DMS-1700094 and DMS-2000128.}
 
 \begin{abstract}
 In this short note, we show that given a cost function $c$, any coupling $\pi$ of two probability measures where the second is a discrete measure can be associated to a certain bipartite graph containing a perfect matching, based on the value of the infinity transport cost $\norm{c}_{L^\infty(\pi)}$. This correspondence between couplings and bipartite graphs is explicitly constructed. We give two applications of this result to the $\mathcal{W}_\infty$ optimal transport problem when the target measure is discrete, the first is a condition to ensure existence of an optimal plan induced by a mapping, and the second is a numerical approach to approximating optimal plans.
 \end{abstract}
\maketitle

\section{Introduction}

\subsection{Problem Statement}
In this paper, we relate the $\mathcal{W}_\infty$-optimal transport problem to a combinatorial matching problem in the case where the target measure is discrete. Our main result is valid for \emph{any} source measure, in particular one which may not be absolutely continuous. As applications, we first obtain a condition ensuring there exists an optimal plan induced by a mapping, and second, a numerical method to approximate optimal plans in the $\mathcal{W}_\infty$-transport problem, which gives the first numerical algorithm for this problem. In this paper, a \emph{discrete measure} will always refer to a \emph{finite} linear combination of delta measures.

We recall the problem as follows. Let $(X, \mu)$ be an arbitrary probability measure space, and $Y = \{y_1, \dots, y_N \}$ be a finite set. We fix some probability measure $\nu$ whose support is equal to $Y$ and some function $c: X\times Y \to \R$ that is measurable with respect to the product $\sigma$-algebra. Additionally, we write $\Pi(\mu, \nu)$ for the collection of probability measures on $X\times Y$ whose left and right marginals equal $\mu$ and $\nu$ respectively.  Then the \emph{$\mathcal{W}_\infty$-optimal transport problem} is to find some $\pi \in \Pi(\mu, \nu)$ so that
\begin{equation}\label{eqn: Kantorovich}
\mathcal{W}^c_\infty(\mu, \nu):=\norm {c}_{L^\infty(\pi)} = \inf_{\ti \pi \in \Pi(\mu, \nu)} \norm {c}_{L^\infty(\ti \pi)}.
\end{equation}
Any $\pi \in \Pi(\mu, \nu)$ is referred to as a transport plan and a  minimizing $\pi$ is referred to as an \emph{$\mathcal{W}_\infty$-optimal transport plan}. We also say that $\norm {c}_{L^\infty(\pi)}$ is the \emph{$\infty$-transport cost} of the transport plan $\pi$. 

Additionally, if a  transport plan $\pi$ is of the form $(\Id \times T)_\# \mu$ for some measurable map $T: X\to Y$, then the map $T$ is called a \emph{transport map} and $\pi$ is \emph{induced by $T$}.

\subsection{Main results}

We will show that existence of a transport plan (not necessarily optimal) with some transport cost can be characterized by finding a perfect matching in a certain bipartite graph, built using the source and target measures $\mu$ and $\nu$. We start with some definitions.

\begin{defin}\label{def: graph def}
	
In this paper, a \emph{bipartite graph} $G$, will refer to a graph with (non-negatively) weighted vertices and unweighted simple edges, which is such that the vertex set can be divided into two disjoint sets $L$ and $R$ (the \emph{left} and \emph{right} vertex sets), and all edges connect exactly one vertex in $L$ with one vertex in $R$.

We refer to $V := L \coprod R$ (the disjoint union) as the \emph{vertex set}. For any $v \in V$ we use $w(v)$ to denote the weight of the vertex $v$.

Furthermore given any subset $S \subset V$, we use $\Gamma(S)$ to denote the neighbors of $S$, i.e. $\Gamma(S)$ is the collection of all $ v \in V$ so that there exists $\ti v \in V$ such that there is an edge between $v$ and $\ti v$. 
\end{defin}

\begin{defin}
	Given a bipartite graph $G$, a \emph{matching} is a map $M: L \times R \to [0, \infty)$. A matching is said to be \emph{valid} if it satisfies the following two conditions. 
\begin{enumerate}
 \item For any $l \in L$ and  $r \in R$, $M(l,r) = 0$ unless there is an edge between $l$ and $r$.
 \item  $\sum_{r' \in R} M(l,r') \leq w(l)$ and $\sum_{l' \in L} M(l',r) \leq w(r)$ for all $l\in L$ and $r\in R$. 
\end{enumerate}
	 Finally we say that a matching is \emph{perfect} if $\sum_{r' \in R} M(l,r') = w(l)$ and $\sum_{l' \in L} M(l',r) = w(r)$ for all $l\in L$ and $r\in R$. 
	
\end{defin}

\begin{defin}\label{def: G_omega}
	
A \emph{transport graph} is a bipartite graph where $L = 2^Y$ and $R = Y$, and there is an edge between $l \in L$ and $r \in R$ if and only if $r \in l$. Furthermore we require that the weights  of vertices in $L$ and $R$ respectively sum to $1$. 

For any $\omega \in \R$, the \emph{$\omega$-transport graph}, denoted $G_\omega$, is a transport graph where the vertex weights for any $y\in Y$ and $A\in 2^Y$ are defined by
\begin{align*}
w(y)&=\nu(\{y\}),\\
w(A)&=\mu(X_A),
\end{align*}
where 
\begin{align}\label{eqn: X_A}
X_A := \bigcap_{y\in A} \{x\mid c(x,y) \leq \omega\} \cap \bigcap_{y \not\in A} \{x\mid c(x,y) > \omega\}.
\end{align}
\end{defin}

\begin{rmk}
We remark that any bipartite graph can be made into a transport graph by labeling the left hand vertices with its collection of neighbors and adding zero weight left vertices for any remaining subsets. Note that in a transport graph, if $M$ is a valid matching then $\sum_{r' \in R} M(l,r') = w(l)$ for all $l\in L$ if and only if  $\sum_{l' \in L} M(l',r) = w(r)$ for all $r\in R$; in particular, either condition implies $M$ is perfect.
\end{rmk}

With this terminology in hand, we can state our main result.
\begin{thm}\label{thm: matching and plan correspondence}
	
Let $\mu\in \mathcal{P}(X)$ and $\nu$ be a discrete measure whose support is the finite set $Y$. Then there exists a transport plan $\pi\in \Pi(\mu, \nu)$ with $\infty$-transport cost at most $\omega$ if and only if the $\omega$-transport graph $G_\omega$ has a perfect matching. Furthermore, if $\mu$  has no atoms and $G_\omega$ has a perfect matching for some $\omega$, such a corresponding transport plan can be taken to arise from a transport map. 
\end{thm}
We note that the proof of the theorem gives explicit constructions of a transport plan / perfect matching arising from this correspondence (see \eqref{eqn: construct matching} and \eqref{eqn: pi def}). Finally, it is a simple matter to obtain the following useful corollary.
\begin{cor}\label{cor: corollary}
 If $\mu$ has no atoms, then for any $\mathcal{W}_\infty$-optimal transport plan $\pi$, there exists a transport map $T$ such that $(\Id\times T)_\#\mu$ has the same $\infty$-transport cost as $\pi$.
\end{cor}
In Section \ref{sec: main} below we give the proofs of these main results, Theorem \ref{thm: matching and plan correspondence} and Corollary \ref{cor: corollary}.

One interesting observation we can make is that \emph{any} bipartite graph can be suitably modified, and realized as the $\omega$-transport graph for a certain $\mathcal{W}_\infty$-optimal transport problem. Since Theorem \ref{thm: matching and plan correspondence} gives explicit constructions to go between transport plans and $\omega$-transport graphs, this shows that solving the $\mathcal{W}_\infty$-transport problem is equivalent to solving the matching problem for an arbitrary bipartite graph. This will be explored in Section \ref{sec: opt bounds}.

Finally we propose an application of Theorem \ref{thm: matching and plan correspondence} in order to numerically find approximations of $\mathcal{W}_\infty$-optimal transport plans. The idea is the following. Fix a desired error tolerance $\epsilon>0$. Then for any $\omega\in (\mathcal{W}_\infty(\mu, \nu), \mathcal{W}_\infty(\mu, \nu)+\epsilon]$, by Theorem \ref{thm: matching and plan correspondence} there exists a perfect matching in the corresponding $\omega$-transport graph. If it is possible to find this matching, then we can obtain a transport plan via \eqref{eqn: pi def} whose $\infty$-transport cost is within $\epsilon$ of the optimal value. This can be exploited as there are well established numerical methods to find a perfect matching in a bipartite graph, if the existence of such a matching is known. In practice, since the actual optimal value $\mathcal{W}_\infty(\mu, \nu)$ is unknown, it is necessary to start with a sufficiently large interval and iteratively do interval halving. In Section \ref{sec: numerics} we detail the numerical algorithm, and present some empirical examples.

\subsection{Literature review}
The $\mathcal{W}_\infty$ problem has appeared in a number of applications, we give a nonexhaustive review of a few examples. The problem was first considered by McCann (see \cite{McCann06}) to analyze a variation formulation for the problem of rotating binary stars. It was later considered by Carrillo, Gualdani, and Toscani in porous medium flow, to bound growth of the wetted region  (\cite{CarrilloGualdaniToscani04}). Finally,  $\mathcal{W}_\infty$ transport has recently appeared in quantitative convergence of empirical measures, and of Gromov-Hausdorff convergence of discrete geometric structures to the smooth one on the torus (\cite{GarciaTrillosSlepcev15, GarciaTrillos20}).

Theoretical aspects of the $\mathcal{W}_{\infty}$ problem for cost given by a power of Euclidean distance are treated in \cite{Champion2008}. There the authors introduce the notion of \emph{infinitely cyclical monotonicity}, and show this condition characterizes optimizers in the $\mathcal{W}_\infty$ problem, this is generalized in \cite{Jylha15} to other cost functions. Additionally, it is shown that if the source measure $\mu$ gives no mass to $n-1$ dimensional Lipschitz sets, and with some mild conditions on $c$, an optimal plan that is infinitely cyclically monotone is induced by a transport map (\cite[Theorem 3.5]{Jylha15}). Our result Corollary \ref{cor: corollary} states that under the weaker assumption that $\mu$ has no atoms, and for arbitrary cost function $c$, \emph{if} there exists an optimal plan then there also exists an optimal plan induced by a map; however note that we do not claim any kind of uniqueness. A dual problem is also treated in \cite{BarronBoceaJensen17}; our methods in this paper use neither duality, nor the notion of infinitely cyclical montonicity.

We also comment, there currently do not appear to be any existing numerical methods for the $\mathcal{W}_\infty$ problem, thus the method presented here is the first to be proposed.

\section{Proofs of main results}\label{sec: main}
In this section we fix an $\omega\in \R$, and take the sets $X_A$ as in \eqref{eqn: X_A}. We first show a basic partitioning property of the $X_A$.
\begin{lem}\label{lem: X_A disjoint partition}
	
The collection $\{X_A\}_{A\in 2^Y}$ is a disjoint partition of $X$, i.e. $X_A \cap X_B = \emptyset$ if $A \neq B$ and $X = \bigcup_{A \in 2^Y} X_A$. 
	
\end{lem}

\begin{proof}	
Fix some $A, B \subset Y$ so that $A \neq B$, then without loss of generality there exists $y \in A$ so that $y \not\in B$. Then by definition
\[
X_A \subset \{x\mid c(x,y) \leq \omega\}
\] 	
and 
\[
X_B \subset \{x\mid c(x,y) > \omega\}
\]
but clearly $\{x\mid c(x,y) \leq \omega\} \cap \{x: c(x,y) > \omega\} = \emptyset$. Hence the $X_A$ are disjoint.

Next to see that the $X_A$ cover $X$, pick any $x \in X$. We define 
\begin{align*}
A:=\{y \in Y\mid c(x,y) \leq \omega\},
\end{align*}
it is then easily seen that $x \in X_A$, even if $A=\emptyset$. 
\end{proof}

\begin{proof}[Proof of Theorem \ref{thm: matching and plan correspondence}]
	
Let $G_\omega$ be the associated $\omega$ transport graph defined using the sets $X_A$, $\mu$, and $\nu$ as in Definition \ref{def: G_omega}. Recall that we write $L = 2^Y$ and $R = Y$ for the left and right vertex sets of $G_\omega$.
	
First let $\pi$ be a transport plan satisfying $\norm{c}_{L^\infty(\pi)} \leq \omega$. Then we can define the matching $M$ by setting 
\begin{align}\label{eqn: construct matching}
M(A, y) = \pi(X_A \times \{y\})
\end{align}
 for any $A \in L$ and $y \in R$. 

We will show that $M$ is a perfect matching. First, if $M(A, y) =\pi(X_A \times \{y\}) > 0$, there exists an $x \in X_A$ so that $c(x, y) \leq \norm{c}_{L^\infty(\pi)} \leq \omega$.  We conclude that $y \in A$, as if $y \not\in A$ we would have $X_A \subset \{\ti x: c(\ti x,y) > \omega\}$. In particular there is an edge between $A$ and $y$ in $G_\omega$. 

Next 
\[
\sum_{A \in L} M(A, y) 
= \sum_{A \in L} \pi(X_A \times \{y\})
= \pi(\bigcup_{A \in L} X_A \times \{y\})
= \pi(X \times \{y\})
= w(y)
\]
where we have used Lemma \ref{lem: X_A disjoint partition} for the middle two equalities. We also see
\[
\sum_{y \in R} M(A, y) 
= \sum_{y \in R} \pi(X_A \times \{y\})
= \pi(X_A \times \bigcup_{y \in R} \{y\})
= \pi(X_A \times Y)
= \mu(X_A)
= w(A).
\]
This completes the proof that $M$ is a perfect matching. 

Next suppose that we are given a perfect matching $M$ in $G_\omega$. We want to construct a transport plan. Note that by Lemma \ref{lem: X_A disjoint partition} the collection $\{X_A \times \{y\}\}_{(A, y)\in L\times R}$ form a partition of $X \times Y$. Define $\pi\in \mathcal{P}(X\times Y)$ as follows. If $\mu(X_A) = 0$ we set $\pi\big|_{X_A \times \{y\}} \equiv 0$. Otherwise we set
\begin{align}\label{eqn: pi def}
\pi\bigg|_{X_A \times \{y\}} := \frac {M(A, y)}{\mu(X_A)\nu(\{y\})} \(\mu\bigg|_{X_A} \otimes \nu\bigg|_{\{y\}}\),
\end{align}
in other words for $S \subset X \times Y$ we have
\begin{equation}\label{eqn: pi definition}
\pi(S) = \sum_{\{A\in 2^Y\mid \mu(X_A) > 0 \}}  \sum_{y\in Y}\frac {M(A, y)}{\mu(X_A)\nu(\{y\})} (\mu\big|_{X_A} \otimes \nu\big|_{\{y\}})(S).
\end{equation}
Note that for any $Q \subset X$
\[
\pi(Q \times Y) 
&= \sum_{\{A\in 2^Y\mid \mu(X_A) > 0 \}}  \sum_{y\in Y} \frac {M(A, y)}{\mu(X_A)\nu(\{y\})} (\mu\big|_{X_A} \otimes \nu\big|_{\{y\}})(Q \times Y) \\
&= \sum_{\{A\in 2^Y\mid \mu(X_A) > 0 \}}  \sum_{y\in Y}\frac {M(A, y)}{\mu(X_A)\nu(\{y\})} \mu(X_A \cap Q) \nu(\{y\})\\
&= \sum_{\{A\in 2^Y\mid \mu(X_A) > 0 \}} \frac {\mu(X_A \cap Q)}{\mu(X_A)}  \sum_{y\in Y} M(A, y) \\
&= \sum_{\{A\in 2^Y\mid \mu(X_A) > 0 \}} \frac {\mu(X_A \cap Q)}{\mu(X_A)}  \mu(X_A) \\
&= \sum_{\{A\in 2^Y\mid \mu(X_A) > 0 \}} {\mu(X_A \cap Q)} \\
&= \mu(Q)
\]
where we have used that $M$ is a perfect matching in order to obtain that $\sum_{y\in Y} M(A, y) = \sum_{y\in R} M(A, y)=w(A)=\mu(X_A)$. 
Next for any $B \subset Y$
\[
\pi(X \times B) 
&= \sum_{\{A \in 2^Y\mid \mu(X_A) > 0\}}  \sum_{y\in Y} \frac {M(A, y)}{\mu(X_A)\nu(\{y\})} (\mu\big|_{X_A} \otimes \nu\big|_{\{y\}})(X \times B) \\ 
&= \sum_{\{A \in 2^Y\mid \mu(X_A) > 0\}}  \sum_{y\in Y} \frac {M(A, y)}{\mu(X_A)\nu(\{y\})} \mu(X_A) \nu(B \cap \{y\}) \\
&=   \sum_{y\in Y} \frac {\nu(B \cap \{y\}) }{\nu(\{y\})} \sum_{\{A \in 2^Y\mid \mu(X_A) > 0\}} M(A, y)  \\
&=   \sum_{y\in Y} \frac {\nu(B \cap \{y\}) }{\nu(\{y\})} \nu(\{y\})  \\
&= \nu(B)
\]
where we have again used that $M$ is a perfect matching. This shows $\pi \in \Pi(\mu, \nu)$. All that is left to do is to verify that $\norm{c}_{L^\infty(\pi)} \leq \omega$. Note that this is the same as saying that $\pi(\{(\ti x, \ti y) \mid c(\ti x, \ti y) > \omega\}) = 0$. Since $\{X_A \times \{ y \}\}_{(A, y)\in L\times R}$ forms a partition of $X \times Y$, it suffices to show that $\pi((X_A \times \{ y \}) \cap \{(\ti x, \ti y) \mid c(\ti x, \ti y) > \omega\}) = 0$ for any $(A, y)\in L\times R$. 

We consider two cases. First if $M(A,y) = 0$ then by definition we have $\pi(X_A \times \{y\}) = 0$ and so of course $\pi((X_A \times \{ y \} )\cap \{(\ti x, \ti y) \mid c(\ti x, \ti y) > \omega\}) = 0$. Second if $M(A,y) > 0$, then since $M$ is a perfect matching we must have $y \in A$. Hence for every $x \in X_A$ we have $c(x,y) \leq \omega$, in other words $X_A \cap \{\ti x\in X\mid c(\ti x, y) > \omega\} = \emptyset$ and so 
\begin{align*}
 \pi((X_A \times \{y\})\cap \{(\ti x, \ti y)\mid c(\ti x, \ti y) > \omega\}) &\leq \pi((X_A \times \{y\}) \cap (\{\ti x\in X\mid c(\ti x, y) > \omega\} \times \{y\}) )=0
\end{align*}
as desired. 

For the last claim, assume that $\mu$ has no atoms and $M$ is a perfect matching of $G_\omega$. Since $\mu(X_A)=\sum_{y\in Y}M(A, y)$, by \cite[215D: Proposition]{Fremlin2003} there exists a partition $\{X_{A,i}\}_{i=1}^{N}$ of each $X_A$ into $N$ sets, satisfying $\mu(X_{A,i}) = M(A, y_i)$ for each $i\in \{1,\ldots, N\}$. Now define $T$ by $T(x) := y_i$ for $x \in X_{A,i}$. 

Recall that if $y_i \in A$ then $c(x, y_i) \leq \omega$ for every $x \in X_A$. Since $\mu(X_{A,i}) = M(A, y_i) = 0$ if $y_i \not\in A$, we see that for $\mu$ almost every $x$, $c(x, T(x)) \leq \omega$. 

Also 
\[
\mu(T\i(\{y_i\})) 
= \mu\(\bigcup_{A \subset Y} X_{A,i}\)
= \sum_{A \subset Y} \mu(X_{A,i})
= \sum_{A \subset Y} M(A, y_i)
= w(y_i)
= \nu(\{y_i\})
\]
and so $(\Id \times T)_\#\mu$ is a valid transport plan with cost at most $\omega$. 
\end{proof}

\begin{rmk}

We remark that the proof of Theorem \ref{thm: matching and plan correspondence} actually gives a bijective correspondence between the collection of perfect matchings in $G_\omega$ and the collection of transport plans with cost at most $\omega$ modulo ``rearrangment'' inside of each cell $X_A$. 

More rigorously: the construction gives a bijective correspondence between the collection of perfect matchings in $G_\omega$, and the collection of equivalence classes of transport plans with cost at most $\omega$, where each class consists of plans of the form given in \eqref{eqn: pi definition} but the measures $\mu\big|_{X_A} \otimes \nu\big|_{\{y\}}$ can be replaced with any measures that share the same marginals. 
	
\end{rmk}
\begin{proof}[Proof of Corollary \ref{cor: corollary}]
If a $\mathcal{W}_\infty$-optimal transport plan exists, the graph $G_\omega$ with $\omega=\mathcal{W}^c_\infty(\mu, \nu)$ contains a perfect matching by the first half of the above theorem, then we may apply the final statement in the theorem above to $G_\omega$.
\end{proof}
\section{Optimality Bounds}\label{sec: opt bounds}

In this section we show that when the cost is a power of a $p$-norm, numerically solving the $\mathcal{W}_\infty$-optimal transport problem with a small error is at least as hard as the determining if a transport graph has a perfect matching. In particular for the square euclidean cost we reduce the problem of finding a perfect matching in a transport graph to numerically solving the $\mathcal{W}_\infty$-optimal transport problem within an error of $\frac 1N$. Indeed note that $\eps(N,2,2) = \frac 1N$ in Proposition \ref{prop: p-norm}.

For this section we will write $X_{A, \omega}$ for
\[
X_{A, \omega} = \bigcap_{y\in A} \{x\mid c(x,y) \leq \omega\} \cap \bigcap_{y \not\in A} \{x\mid c(x,y) > \omega\}.
\]	
This is the same $X_A$ as in Definition \ref{def: G_omega}, however we will be varying $\omega$ in this section and so we add it to our notation. 

%
%

\begin{prop}\label{prop: arbitrary graph}
Let $\omegaParameters \subset \R$ and $c, X, Y$ be such that
\[
\bigcap_{\omega \in \omegaParameters} X_{A, \omega} \neq \emptyset
\]
for every $A \subset Y$. 

Then for every transport graph $G$, there exists a pair of probability measures $(\mu, \nu)$ so that $G=G_\omega$ for every $\omega \in \omegaParameters$ where $G_\omega$ is the transportation graph defined using $(\mu, \nu)$ in Definition \ref{def: G_omega}. 
\end{prop}

\begin{proof}	
Fix a transport graph $G$ with vertex weight function $w$. For each $A \subset Y$ choose a point $x_A \in \bigcap_{\omega \in \omegaParameters} X_{A, \omega}$, then define $\mu$ by $\mu = \sum_{A \subset Y} w(A) \delta_{x_A}$ and $\nu$ by $\nu = \sum_{y \in Y} w(y) \delta_{y}$. 
 Since for each $\omega\in \Lambda$, $\{X_{A, \omega}\}_{A\in 2^Y}$ is a disjoint collection  by Lemma \ref{lem: X_A disjoint partition} and we have $x_A \in X_{A, \omega}$, we see $\mu(X_{A, \omega}) = w(A)$ and so $G_\omega = G$. 
\end{proof}

\begin{prop}\label{prop: matching dichotomy}
Suppose $G$ is a transport graph, $\omegaParameters \subset \R$, and $\mu\in \mathcal{P}(X)$, $\nu\in \mathcal{P}(Y)$ are measures such that $G_\omega = G$ for every $\omega \in \omegaParameters$. Then 
\begin{enumerate}
\item   $\inf_{\ti \pi \in \Pi(\mu, \nu)} \norm {c}_{L^\infty(\ti \pi)} \leq \inf \omegaParameters$ if and only if $G$ has a perfect matching
\item $\inf_{\ti \pi \in \Pi(\mu, \nu)} \norm {c}_{L^\infty(\ti \pi)} \geq \sup \omegaParameters$ if and only if $G$ does not have a perfect matching,
\end{enumerate}
hence in all cases
\begin{align*}
 \inf_{\ti \pi \in \Pi(\mu, \nu)} \norm {c}_{L^\infty(\ti \pi)} \in (-\infty, \inf\omegaParameters]\cup [\sup\omegaParameters, \infty).
\end{align*}
In particular it suffices to solve \eqref{eqn: Kantorovich} with this choice of $\mu$ and $\nu$ to an error of less than $\frac{\diam \omegaParameters}2$ in order to determine if $G$ has a perfect matching. 	
\end{prop}

\begin{proof}	
Suppose that $G$ has a perfect matching. Then for every $\omega \in \omegaParameters$, by Theorem \ref{thm: matching and plan correspondence} there exists a transport plan $\pi\in \Pi(\mu, \nu)$ so that $\norm {c}_{L^\infty(\pi)} \leq \omega$, hence $\inf_{\ti \pi \in \Pi(\mu, \nu)} \norm {c}_{L^\infty(\ti \pi)} \leq \inf \omegaParameters$.

Now suppose that $G$ does not have a perfect matching. Then for every $\omega \in \omegaParameters$, again by Theorem \ref{thm: matching and plan correspondence} there cannot exist any transport plan $\pi\in \Pi(\mu, \nu)$ with $\norm {c}_{L^\infty(\pi)} \leq \omega$, hence $\inf_{\ti \pi \in \Pi(\mu, \nu)} \norm {c}_{L^\infty(\ti \pi)} \geq \omega$. In particular we obtain $\inf_{\ti \pi \in \Pi(\mu, \nu)} \norm {c}_{L^\infty(\ti \pi)} \geq \sup \omegaParameters$. 

For the final claim, any interval of length $\eps < \diam \omegaParameters$ can only intersect one of $(-\infty, \inf \omegaParameters]$ or $[\sup \omegaParameters, \infty)$. Thus determining $\inf_{\ti \pi \in \Pi(\mu, \nu)} \norm {c}_{L^\infty(\ti \pi)}$ to within an error of $\epsilon$ will indicate which of the two cases above we are in, and hence if there is a perfect matching or not.
\end{proof}

\begin{prop}\label{prop: p-norm}

Let $X = \R^N$, $y_i = e_i$ for $i \in \{1, \dots, N \}$, $q>0$, $p>1$, and $c = \norm{\cdot}_p^q$. Then the hypotheses of Proposition \ref{prop: matching dichotomy} are satisfied with $\omegaParameters = (1 - \eps(N,p,q) ,1)$ where 
\[
\eps(N,p,q) := 1 - \bigg((1- (1+(N-1)^{\frac 1 {p-1}})^{-1})^p + (N-1)(1+(N-1)^{\frac 1 {p-1}})^{-p} \bigg)^{q/p} > 0.
\] 
In other words
\[
\bigcap_{\omega \in  (1 - \eps(N, p, q),1)} X_{A, \omega} \neq \emptyset
\]
for every $A \subset Y$.

\end{prop}

\begin{proof}
Set $\alpha := (1+(N-1)^{\frac 1 {p-1}})\i$ and for each $A\subset Y$, define $x_A \in \R^N$ by
\[
x_A^i :=
\begin{cases}
\alpha, & \text{ if } y_i \in A,\\
0, & \text{ else.}
\end{cases}
\]
We claim that $x_A \in 	\bigcap_{\omega \in \omegaParameters} X_{A, \omega}$. First fix some $y_k=e_k \not\in A$. Then
\[
\norm{x_A - y_k}_p^q = (1 + \alpha^p {\abs{A}})^{q/p}  \geq 1
\]
and so $c(x_A, y_k) \geq \omega$ for every $\omega \in (1 - \eps(N,p,q) ,1)$. Next fix some $y_k=e_k \in A$. We have
\[
\norm{x_A - y_k}_p^q
= \bigg((1- \alpha)^p + \alpha^p {(\abs{A}-1)} \bigg)^{q/p}
\leq \bigg((1- \alpha)^p + \alpha^p {(N-1)} \bigg)^{q/p}
= 1 - \eps(p,q,N)
\]
and so $c(x_A, y_k) \leq \omega$ for every $\omega \in \omegaParameters$. This shows that $x_A \in \bigcap_{\omega \in \omegaParameters} X_{A, \omega}$ as desired. 

We note that $\eps(N,p,q) > 0$ since $\alpha$ is the minimizer of of the function $g(t) := (1- t)^p + t^p {(N-1)}$ over $t \in [0,1]$. Since $p>1$, it is not hard to see that $g'<0$ near $0$, hence $g(t) < g(0)=1$ when $t < 1$ is very close to $0$, thus we obtain $g(\alpha)^{q/p} < 1$. 
\end{proof}


%
%
%

%
%

\section{Numerical examples}\label{sec: numerics}

\subsection{Description of Algorithm}
The proposed algorithm is a bisection algorithm that estimates the value of the optimal $\infty$-transport cost, and then produces an approximation of the solution to the decision problem. Suppose the optimal cost $\mathcal{W}^c_\infty(\mu, \nu)$ is known to lie in some interval $[\omega_1, \omega_2]$. We then query a decision algorithm see if it possible to produce a plan with cost less than $\frac{\omega_1 + \omega_2}{2}$, or in other words, whether $\mathcal{W}^c_\infty(\mu, \nu)$ lies in the upper or lower half of the interval $[\omega_1, \omega_2]$. We then divide the interval $[\omega_1, \omega_2]$ in half, and recursively continue the process until we reach a plan whose transport cost is within some specified error tolerance of the true value $\mathcal{W}^c_\infty(\mu, \nu)$. Note that if $c$ is bounded (which we will assume for the remainder of the paper), we may always begin with the choice $[\omega_1, \omega_2]=[\min c, \max c]$.

\begin{prop}\label{prop: hall matching}
	
If $G$ is a transport graph, it has a perfect matching if and only if for every $A \in L$, $\sum_{l\in A} w(l) \leq \sum_{r \in \Gamma(A)} w(r)$ (recall Definition \ref{def: graph def}).
\end{prop}

\begin{proof}	
This is a version of Hall's theorem and is essentially the same as the first proof of \cite[Section III.3, Theorem 7]{Bollobas1998}. We interpret $\abs{S}$ as the the sum of the weights in $S$, use the version of the max-flow min-cut problem in \cite[Section III.1, Theorem 4]{Bollobas1998}, and note for a transport graph, Bollob{\'a}s's notional of complete matching implies perfect matching. 
\end{proof}

\begin{algorithm}[H]
	\DontPrintSemicolon
	
	\KwIn{A transport graph $G$.}
	\KwOut{True or False}

	\For{$A \subset R$}
	{
	 \If{$\sum_{l\in A} w(l) > \sum_{r \in \Gamma(A)} w(r)$, } 
		{			
	\Return{False} \;		}
	}
	 
	\Return{True}

	\caption{Maximal Matching Algorithm, Hall Matching}
	\label{alg: Maximal Matching Algorithm, Hall Matching}
\end{algorithm}
By Proposition \ref{prop: hall matching} and Theorem \ref{thm: matching and plan correspondence},  Algorithm \ref{alg: Maximal Matching Algorithm, Hall Matching} can be used as the decision algorithm in the binary search process mentioned above. Once we find $\omega$ that is sufficiently close to the optimal value, $\inf_{\ti \pi \in \Pi(\mu, \nu)} \norm {c}_{L^\infty(\ti \pi)}$, we use the Edmonds-Karp algorithm to compute a maximal matching in $G_{\omega}$. Finally from this maximal matching we obtain a transport plan via the method of the proof of Theorem \ref{thm: matching and plan correspondence}.

We remark that Algorithm \ref{alg: Maximal Matching Algorithm, Hall Matching} terminates in $2^N$ steps and that when applied to $G_\omega$ the Edmonds-Karp algorithm terminates in at most $O(N4^N)$ steps, see \cite[Theorem 26.8]{Cormen09}.

\subsection{Numerical Experiments}
In all of the following numerical examples, the source measure $\mu$ is equal to Lebesgue measure (normalized to unit mass) restricted to the square $X=[0, 4]^2\subset 
\R^2$ and the cost function used is $c(x, y)=\lVert x-y\rVert_\infty$. The target will consist of a finite collection of points $Y=\{y_1, \ldots, y_N\}\subset X$ for some $N$. All code has been made  publicly available\footnote{\url{https://github.com/mohit-bansil/W_infinity_2D}}.

For each example below, Figures \ref{figure1-1}, \ref{figure2-1}, and \ref{figure3-1} are graphical representations of the measures
\begin{align*}
\mu_i:= \sum_{A\in 2^Y}\frac{\pi(X_A\times \{y_i\})}{\mu(X_A)\nu(\{y_i\})}\mu\big\vert_{X_A},
\end{align*}
for each point $y_i\in Y$, where $\pi\in \mathcal{P}(X\times Y)$ is the approximate optimal plan produced by the algorithm, the sets $X_A$ are defined as in \eqref{eqn: X_A}, and the quantity $\frac{\pi(X_A\times \{y_i\})}{\mu(X_A)\nu(\{y_i\})}$ is interpreted as $0$ if $\mu(X_A)=0$ (see also \eqref{eqn: pi def}). Effectively, $\mu_i$ is the distribution of mass that is sent to the location $y_i$ under the plan $\pi$.

Figures \ref{figure1-2}, \ref{figure2-2}, and \ref{figure3-2} give the sets $X_A$ for each subset $A\in 2^Y$. Empty cells are displayed in Examples \ref{ex1} and \ref{ex2}, but are excluded in Example \ref{ex3} due to the large number of cells.

In all three examples, the algorithm is run to an upper bound on the error of
\begin{align*}
\abs{\norm{c}_{L^\infty(\pi)}-\mathcal{W}^c_\infty(\mu, \nu)}<10^{-6},
\end{align*}
and all three examples terminate after 24 iterations of Algorithm \ref{alg: Maximal Matching Algorithm, Hall Matching} above.

\begin{ex}\label{ex1}
$y_1=(0, 0)$, $y_2=(0, 4)$, $y_3=(4, 0)$, $y_4=(2, 2)$. $$\nu=0.25(\delta_{y_1}+\delta_{y_2}+\delta_{y_3}+\delta_{y_4})$$
\end{ex}

\begin{figure}[H]
\centering
\begin{mdframed}
\begin{minipage}{0.1\linewidth}
\begin{subfigure}{.05\textwidth}

\begin{tikzpicture}
\pgfplotsset{
colormap={revblackwhite}{gray(0cm)=(1); gray(1cm)=(0)}
}
\pgfplotscolorbardrawstandalone[
    colorbar,
    point meta min=0,
    point meta max=1,
    colorbar style={
        width=0.5cm}
        ]
\end{tikzpicture}%

\end{subfigure}%
\end{minipage}
\begin{minipage}{0.9\linewidth}
\begin{subfigure}{.45\textwidth}
    \centering
    \frame{
    \includegraphics[width=0.45\textwidth]{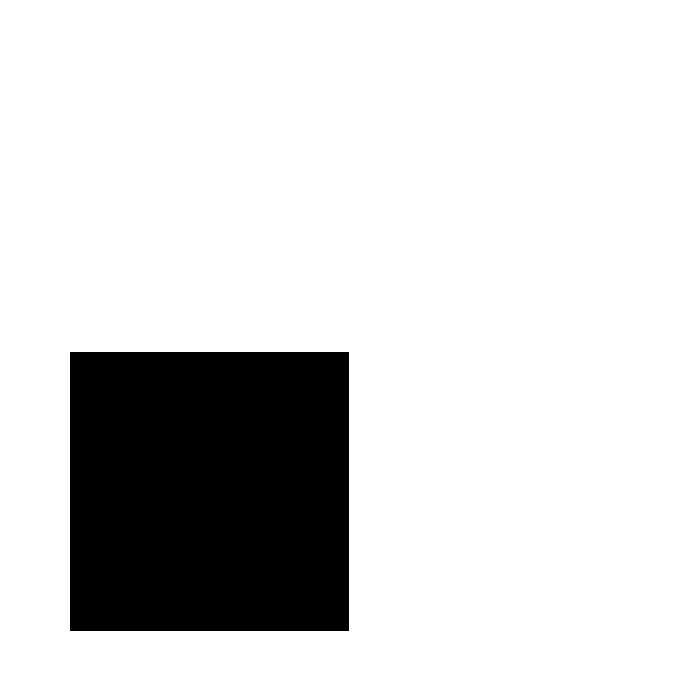}}
    \caption[short]{$\mu_1$}
\end{subfigure}%
\begin{subfigure}{.45\textwidth}
    \centering
    \frame{
    \includegraphics[width=0.45\textwidth]{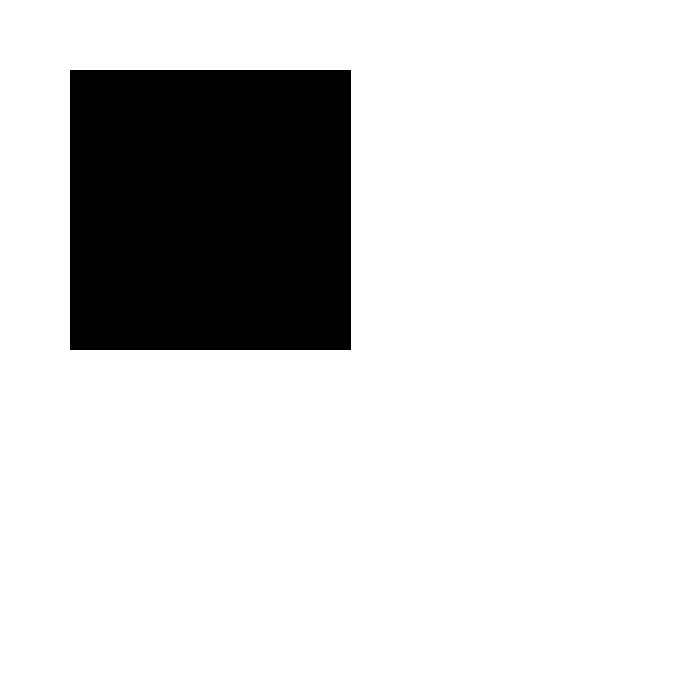}}
    \caption[short]{$\mu_2$}
\end{subfigure}
\begin{subfigure}{.45\textwidth}
    \centering
    \frame{
    \includegraphics[width=0.45\textwidth]{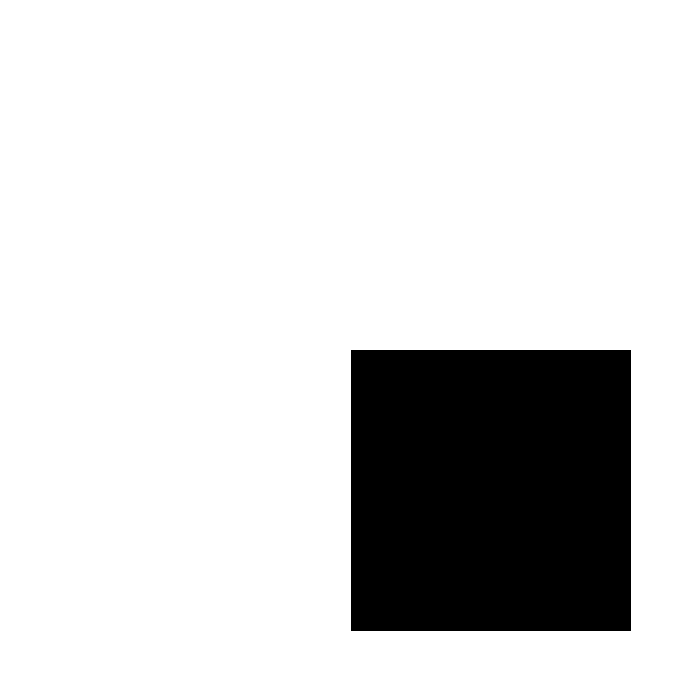}}
    \caption[short]{$\mu_3$}
\end{subfigure}%
\begin{subfigure}{.45\textwidth}
    \centering
    \frame{
    \includegraphics[width=0.45\textwidth]{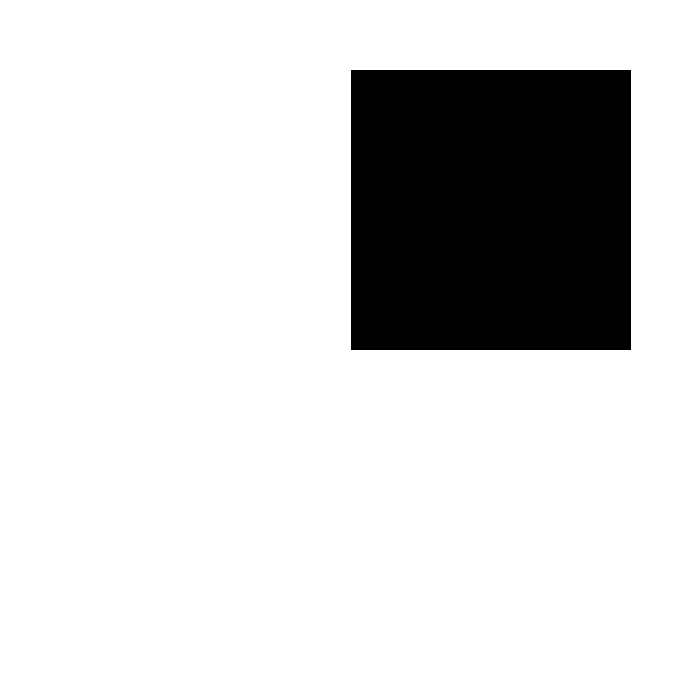}}
    \caption[short]{$\mu_4$}
\end{subfigure}
\end{minipage}
\end{mdframed}
\caption{Transportation of mass: Example \ref{ex1}}\label{figure1-1}
\end{figure}

\begin{figure}[H]
\begin{mdframed}
\centering
\begin{subfigure}{.25\textwidth}
    \centering
    \frame{
    \includegraphics[width=0.25\textwidth]{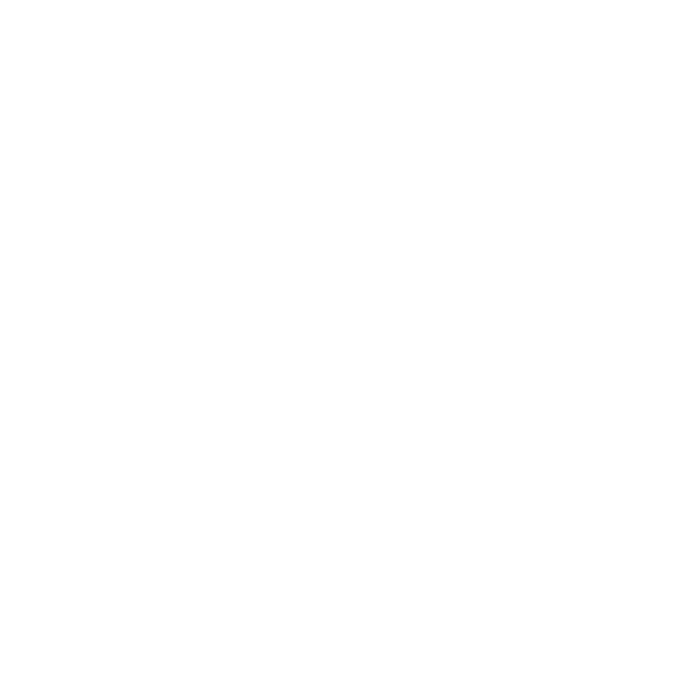}}
    \caption[short]{$A=\{y_1\}$}
\end{subfigure}%
\begin{subfigure}{.25\textwidth}
    \centering
    \frame{
    \includegraphics[width=0.25\textwidth]{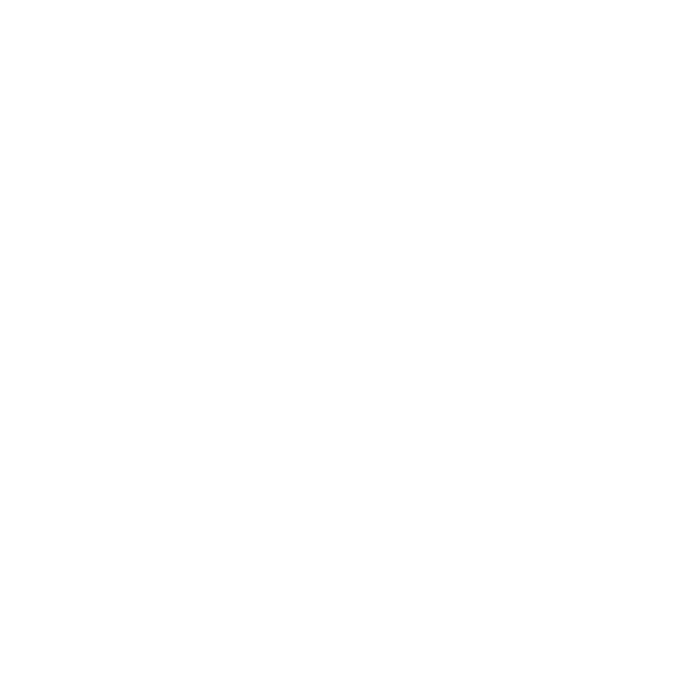}}
    \caption[short]{$A=\{y_2\}$}
\end{subfigure}%
\begin{subfigure}{.25\textwidth}
    \centering
    \frame{
    \includegraphics[width=0.25\textwidth]{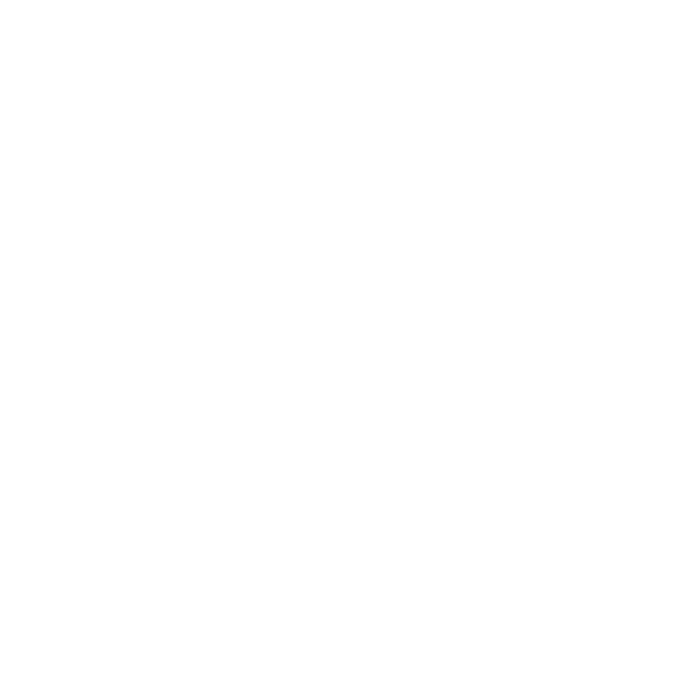}}
    \caption[short]{$A=\{y_3\}$}
\end{subfigure}%
\begin{subfigure}{.25\textwidth}
    \centering
    \frame{
    \includegraphics[width=0.25\textwidth]{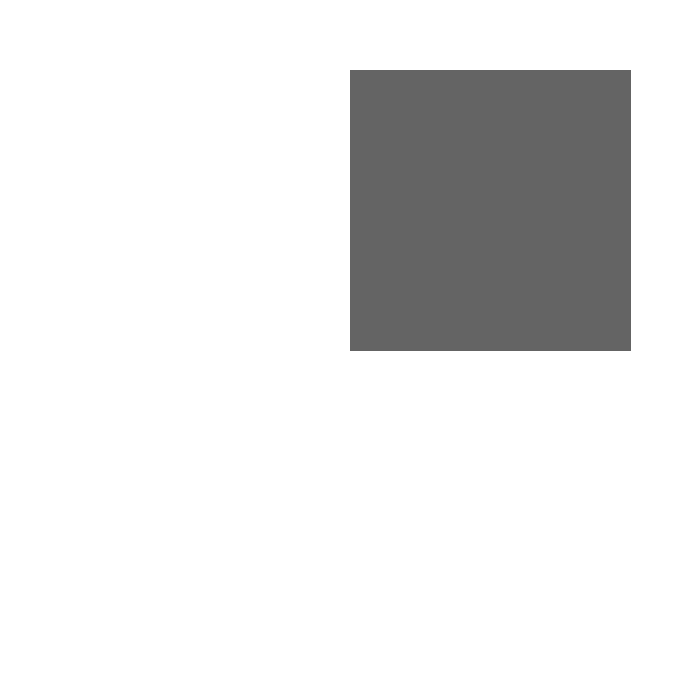}}
    \caption[short]{$A=\{y_4\}$}
\end{subfigure}
\begin{subfigure}{.25\textwidth}
    \centering
    \frame{
    \includegraphics[width=0.25\textwidth]{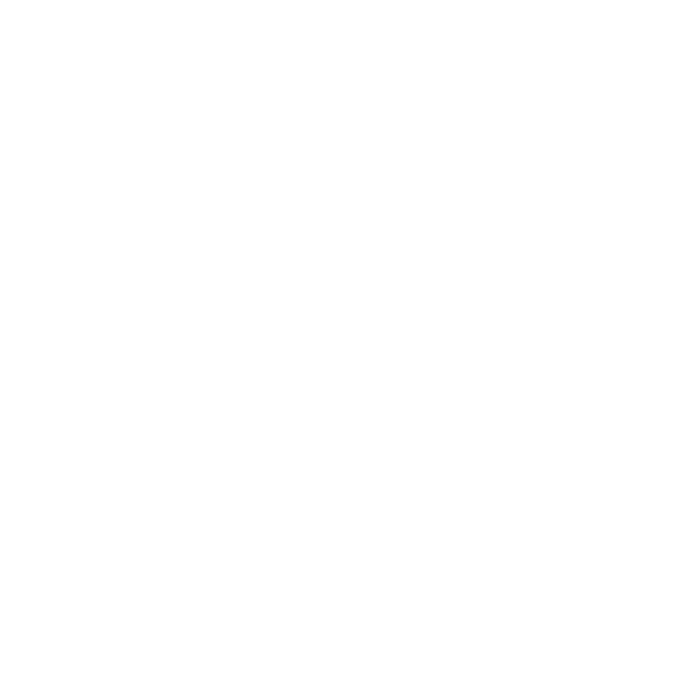}}
    \caption[short]{$A=\{y_1, y_2\}$}
\end{subfigure}%
\begin{subfigure}{.25\textwidth}
    \centering
    \frame{
    \includegraphics[width=0.25\textwidth]{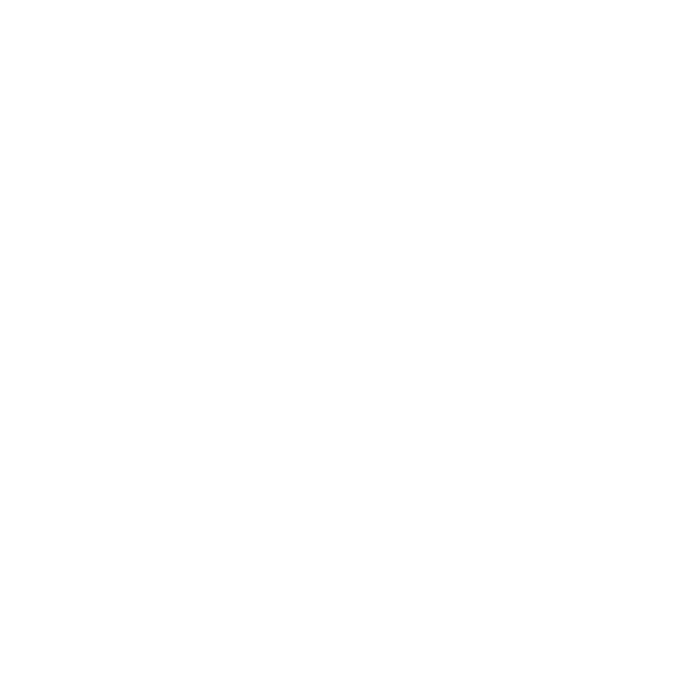}}
    \caption[short]{$A=\{y_1, y_3\}$}
\end{subfigure}%
\begin{subfigure}{.25\textwidth}
    \centering
    \frame{
    \includegraphics[width=0.25\textwidth]{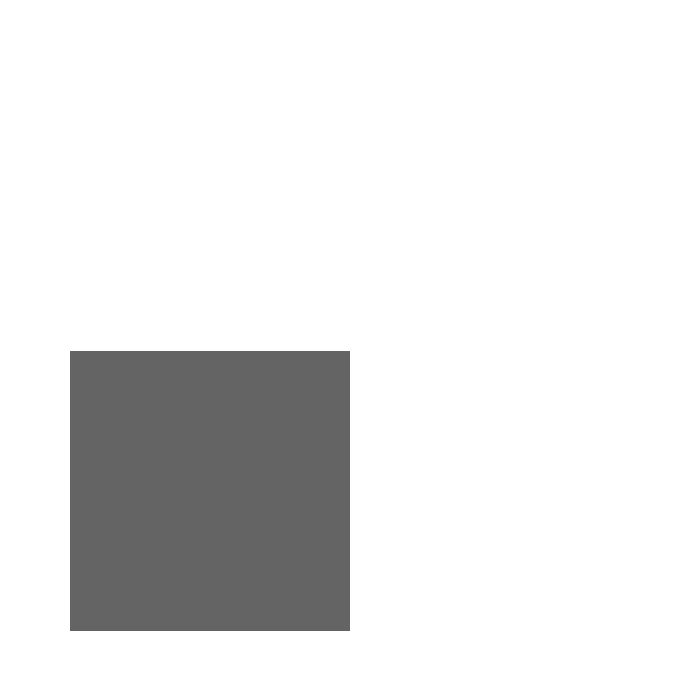}}
    \caption[short]{$A=\{y_1, y_4\}$}
\end{subfigure}%
\begin{subfigure}{.25\textwidth}
    \centering
    \frame{
    \includegraphics[width=0.25\textwidth]{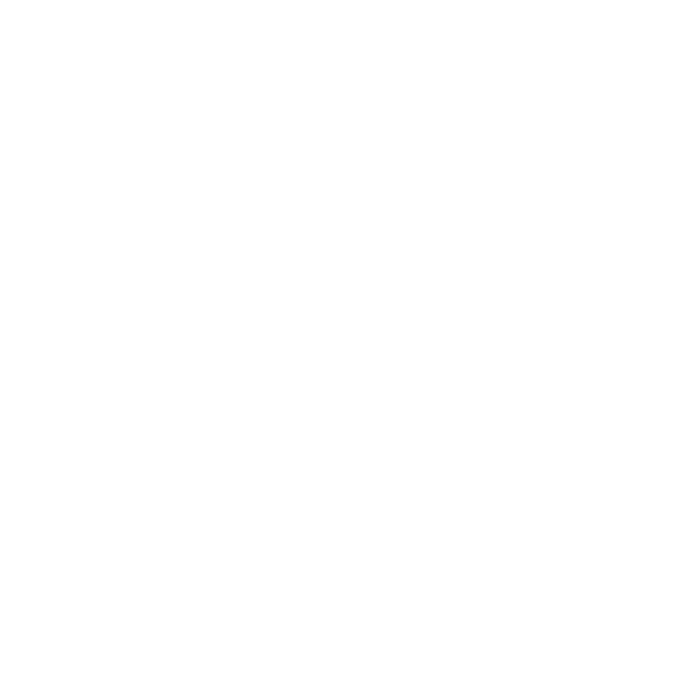}}
    \caption[short]{$A=\{y_2, y_3\}$}
\end{subfigure}
\begin{subfigure}{.25\textwidth}
    \centering
    \frame{
    \includegraphics[width=0.25\textwidth]{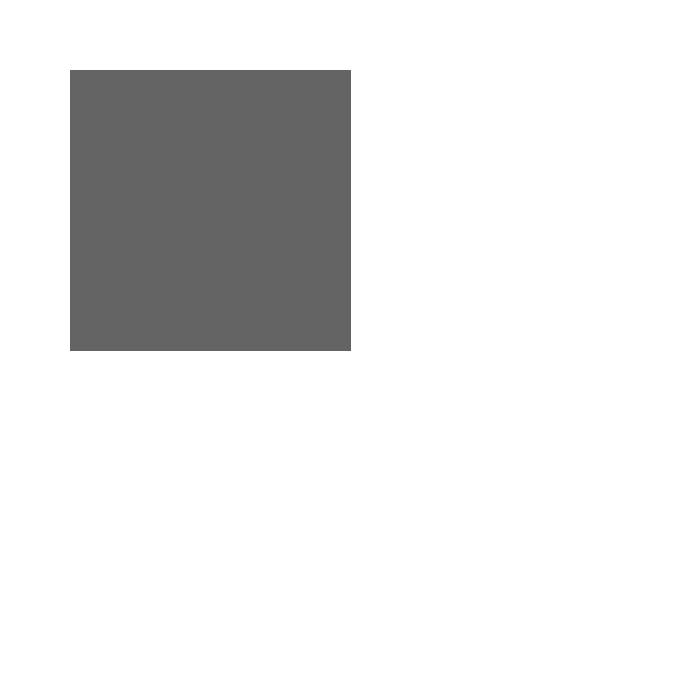}}
    \caption[short]{$A=\{y_2, y_4\}$}
\end{subfigure}%
\begin{subfigure}{.25\textwidth}
    \centering
    \frame{
    \includegraphics[width=0.25\textwidth]{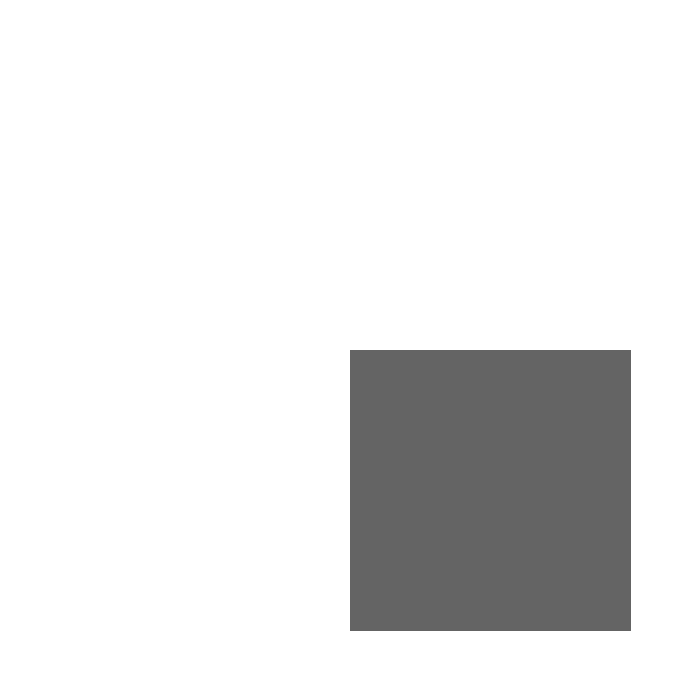}}
    \caption[short]{$A=\{y_3, y_4\}$}
\end{subfigure}%
\begin{subfigure}{.25\textwidth}
    \centering
    \frame{
    \includegraphics[width=0.25\textwidth]{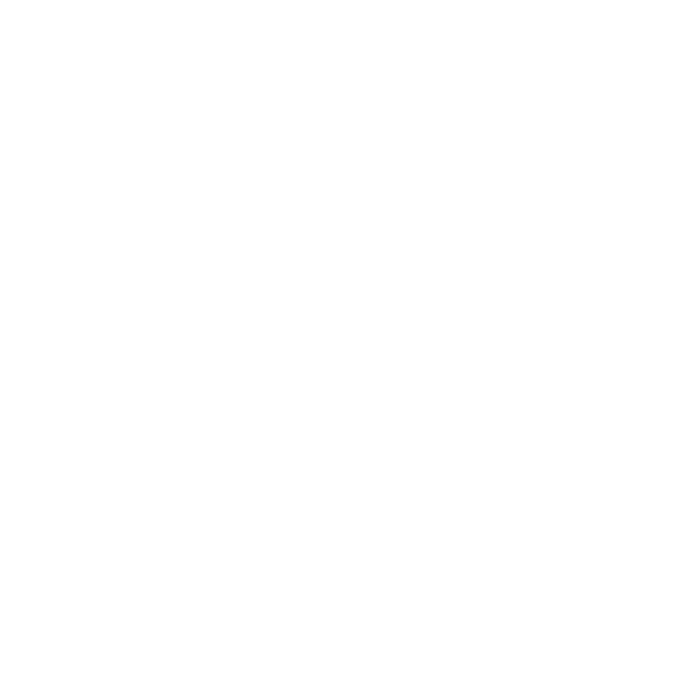}}
    \caption[short]{$A=\{y_1, y_2, y_3\}$}
\end{subfigure}%
\begin{subfigure}{.25\textwidth}
    \centering
    \frame{
    \includegraphics[width=0.25\textwidth]{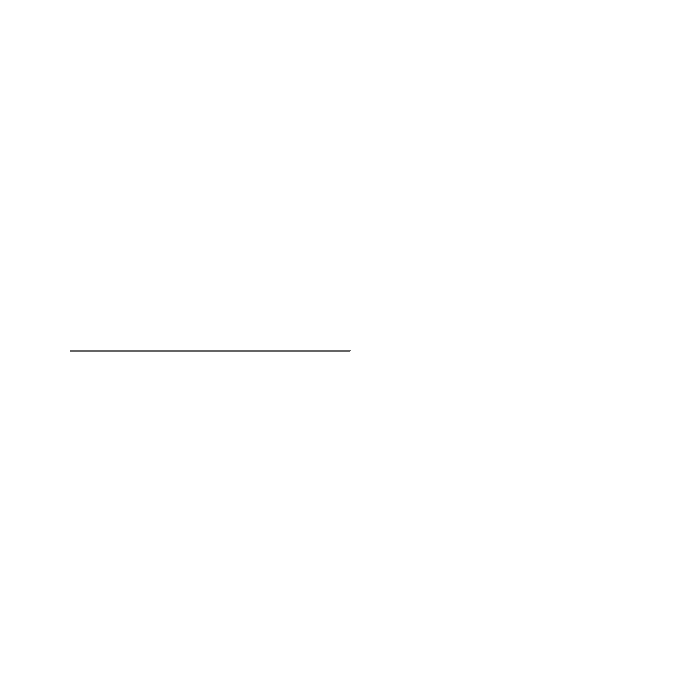}}
    \caption[short]{$A=\{y_1, y_2, y_4\}$}
\end{subfigure}
\raggedright
\begin{subfigure}{.25\textwidth}
    \centering
    \frame{
    \includegraphics[width=0.25\textwidth]{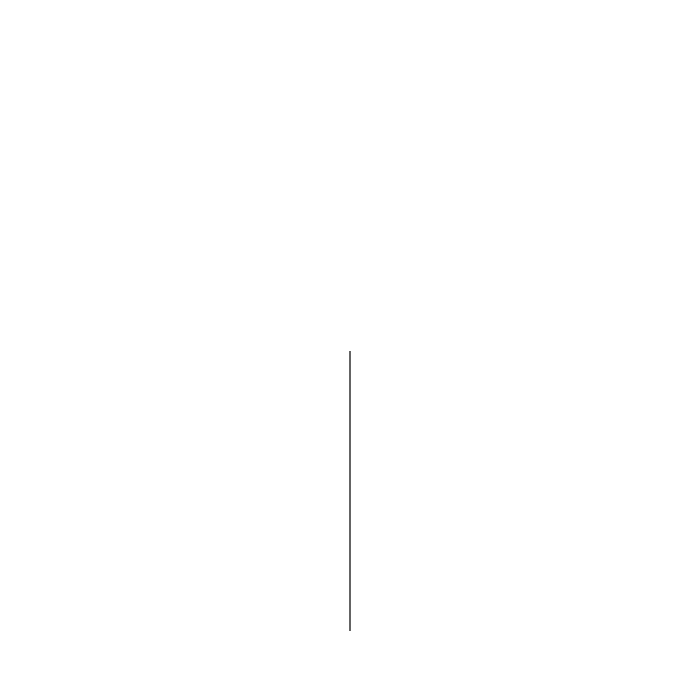}}
    \caption[short]{$A=\{y_1, y_3, y_4\}$}
\end{subfigure}%
\begin{subfigure}{.25\textwidth}
    \centering
    \frame{
    \includegraphics[width=0.25\textwidth]{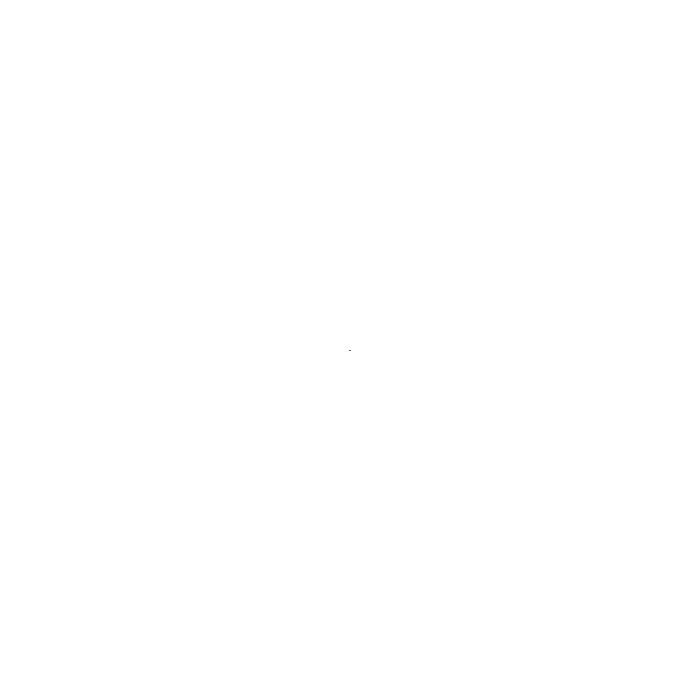}}
    \caption[short]{$A=\{y_2, y_3, y_4\}$}
\end{subfigure}%
\begin{subfigure}{.25\textwidth}
\centering
\frame{
    \includegraphics[width=0.25\textwidth]{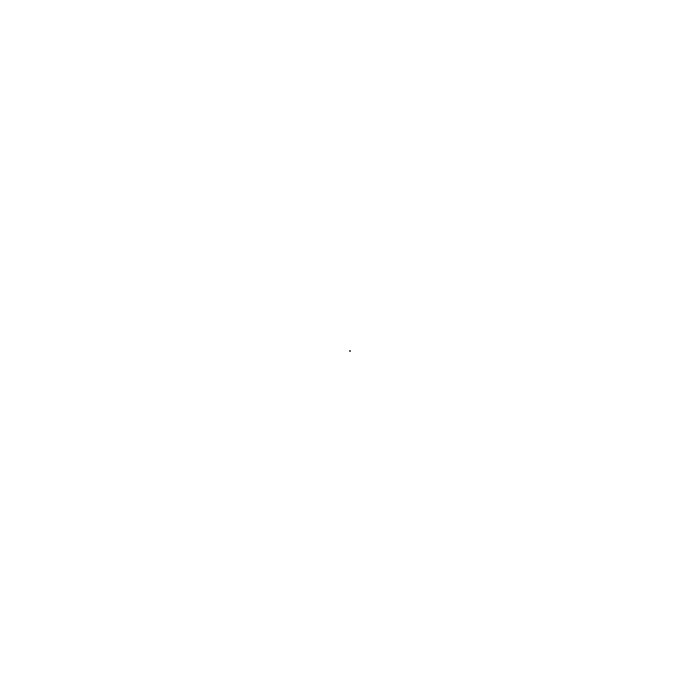}}
    \caption[short]{$A=\{y_1, y_2, y_3, y_4\}$}
\end{subfigure}%
\end{mdframed}
\caption{The cells $X_A$: Example \ref{ex1}}\label{figure1-2}
\end{figure}
\begin{ex}\label{ex2}
The points $y_1,\ldots, y_4$ are taken the same as Example \ref{ex1}, $$\nu=(0.1)\delta_{y_1}+(0.2)\delta_{y_2}+(0.4)\delta_{y_3}+(0.3)\delta_{y_4}$$
\end{ex}

\begin{figure}[H]
\centering
\begin{mdframed}
\begin{minipage}{0.1\linewidth}
\begin{subfigure}{.05\textwidth}

\begin{tikzpicture}
\pgfplotsset{
colormap={revblackwhite}{gray(0cm)=(1); gray(1cm)=(0)}
}
\pgfplotscolorbardrawstandalone[
    colorbar,
    point meta min=0,
    point meta max=1,
    colorbar style={
        width=0.5cm}
        ]
\end{tikzpicture}%

\end{subfigure}%
\end{minipage}
\begin{minipage}{0.9\linewidth}
\begin{subfigure}{.45\textwidth}
    \centering
    \frame{
    \includegraphics[width=0.45\textwidth]{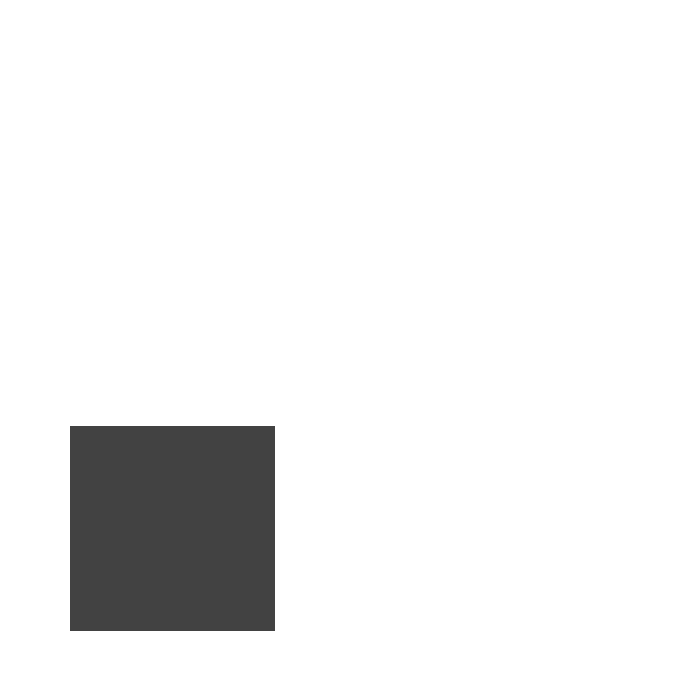}}
    \caption[short]{$\mu_1$}
\end{subfigure}%
\begin{subfigure}{.45\textwidth}
    \centering
    \frame{
    \includegraphics[width=0.45\textwidth]{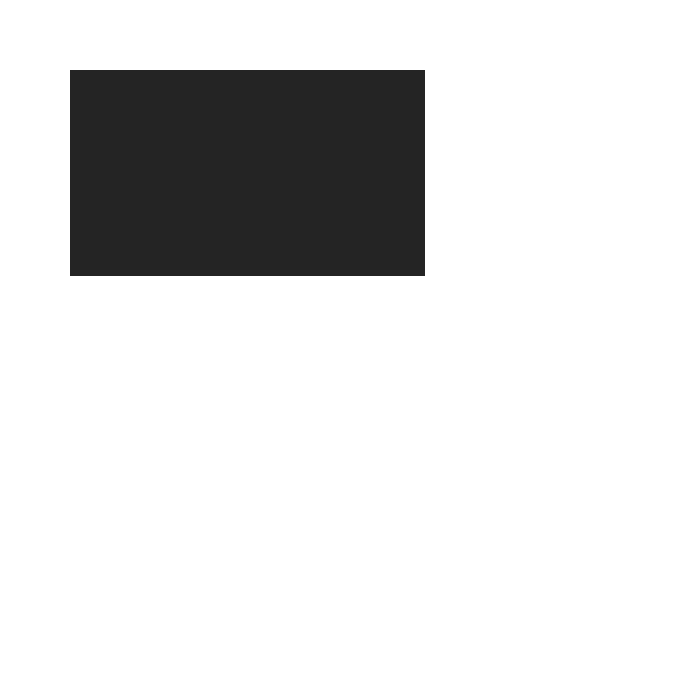}}
    \caption[short]{$\mu_2$}
\end{subfigure}
\begin{subfigure}{.45\textwidth}
    \centering
    \frame{
    \includegraphics[width=0.45\textwidth]{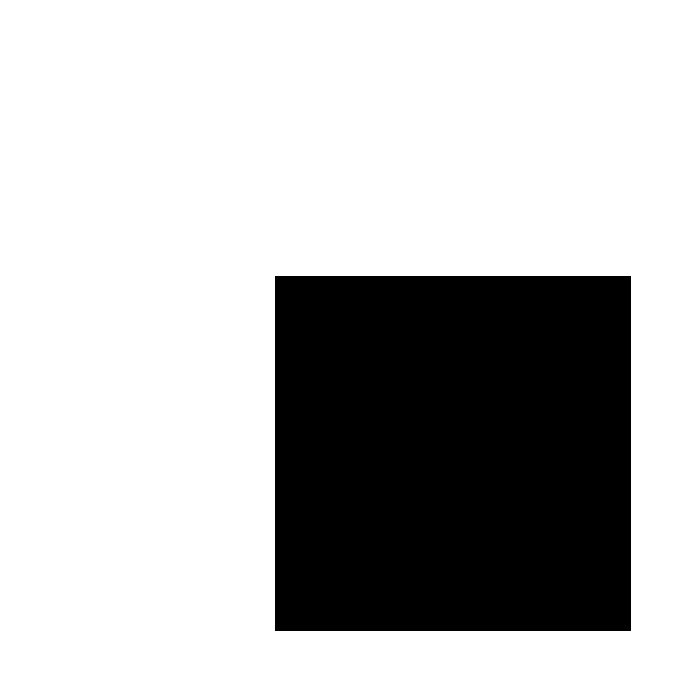}}
    \caption[short]{$\mu_3$}
\end{subfigure}%
\begin{subfigure}{.45\textwidth}
    \centering
    \frame{
    \includegraphics[width=0.45\textwidth]{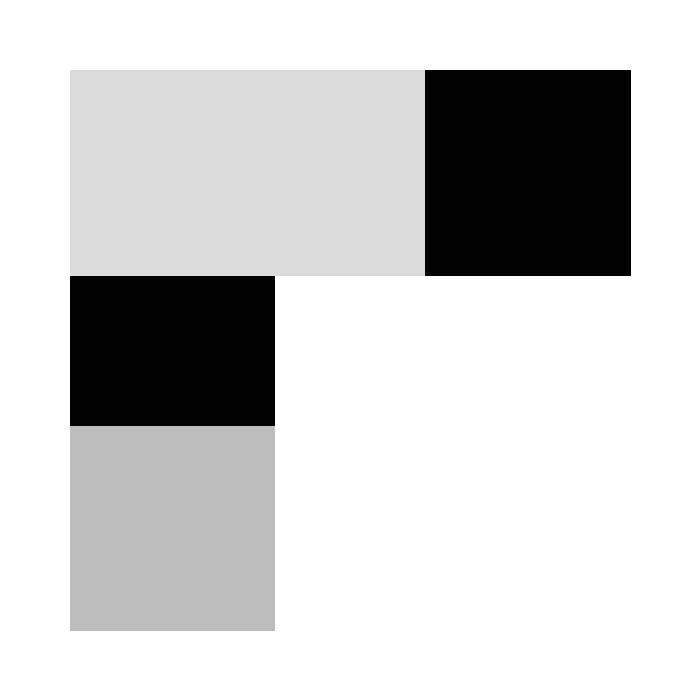}}
    \caption[short]{$\mu_4$}
\end{subfigure}
\end{minipage}
\end{mdframed}
\caption{Transportation of mass: Example \ref{ex2}}\label{figure2-1}
\end{figure}

\begin{figure}[H]
\begin{mdframed}
\centering
\begin{subfigure}{.25\textwidth}
    \centering
    \frame{
    \includegraphics[width=0.25\textwidth]{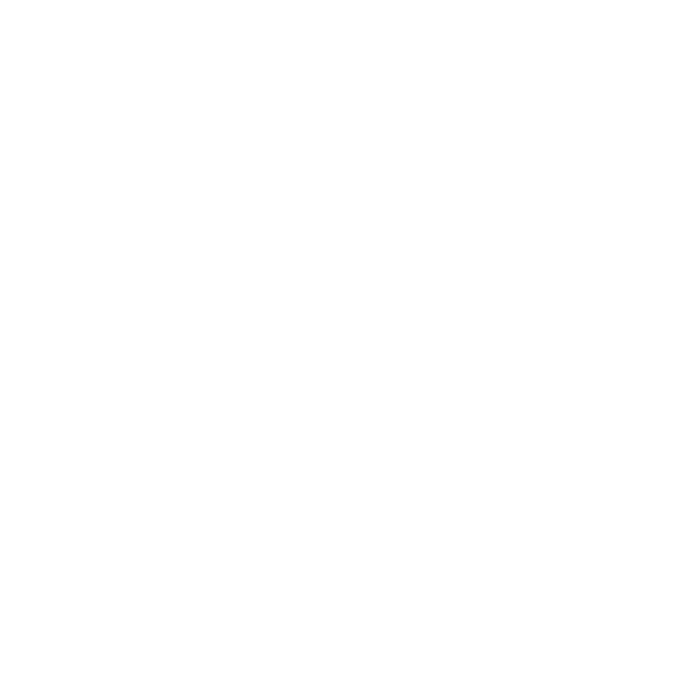}}
    \caption[short]{$A=\{y_1\}$}
\end{subfigure}%
\begin{subfigure}{.25\textwidth}
    \centering
    \frame{
    \includegraphics[width=0.25\textwidth]{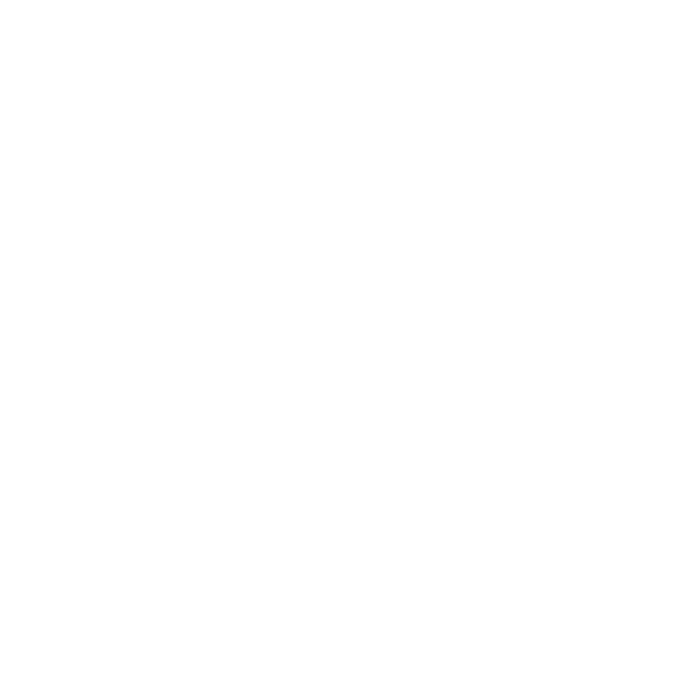}}
    \caption[short]{$A=\{y_2\}$}
\end{subfigure}%
\begin{subfigure}{.25\textwidth}
    \centering
    \frame{
    \includegraphics[width=0.25\textwidth]{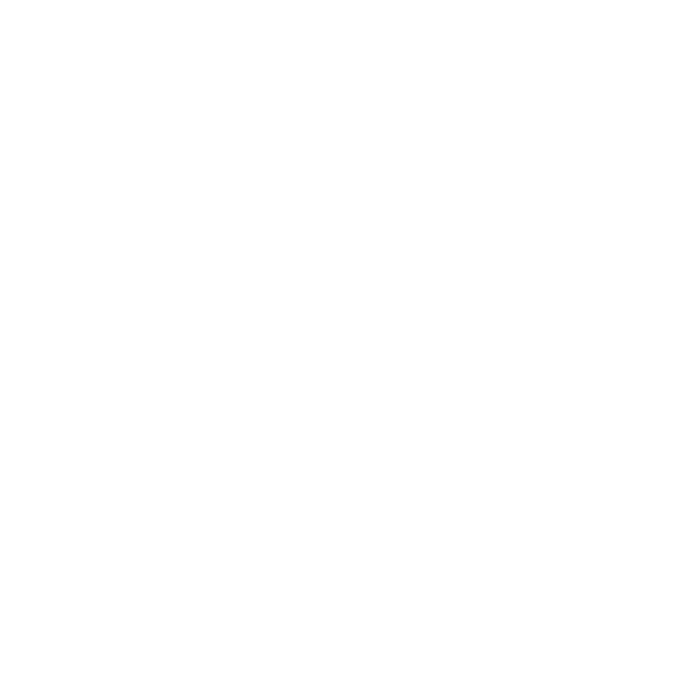}}
    \caption[short]{$A=\{y_3\}$}
\end{subfigure}%
\begin{subfigure}{.25\textwidth}
    \centering
    \frame{
    \includegraphics[width=0.25\textwidth]{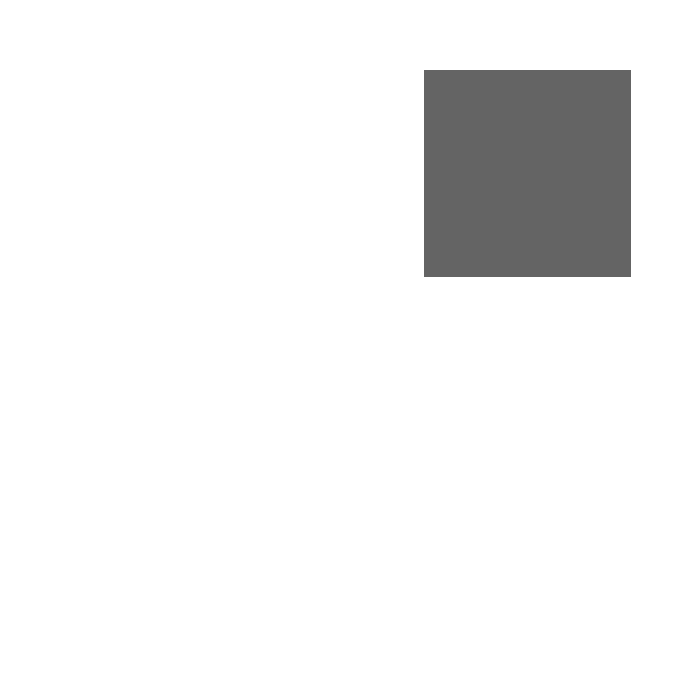}}
    \caption[short]{$A=\{y_4\}$}
\end{subfigure}
\begin{subfigure}{.25\textwidth}
    \centering
    \frame{
    \includegraphics[width=0.25\textwidth]{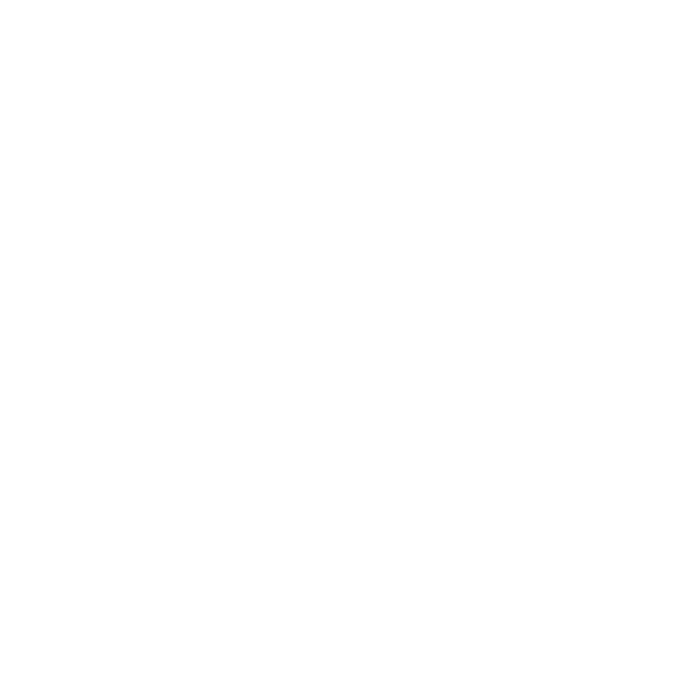}}
    \caption[short]{$A=\{y_1, y_2\}$}
\end{subfigure}%
\begin{subfigure}{.25\textwidth}
    \centering
    \frame{
    \includegraphics[width=0.25\textwidth]{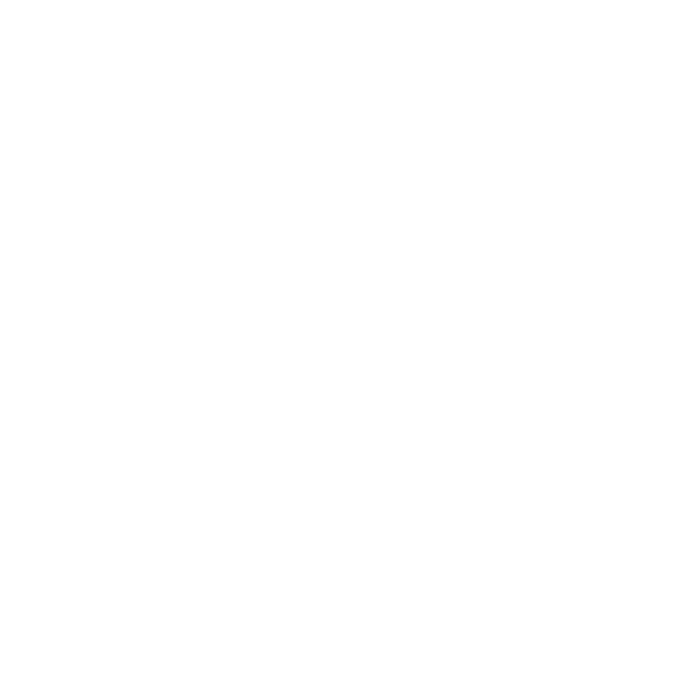}}
    \caption[short]{$A=\{y_1, y_3\}$}
\end{subfigure}%
\begin{subfigure}{.25\textwidth}
    \centering
    \frame{
    \includegraphics[width=0.25\textwidth]{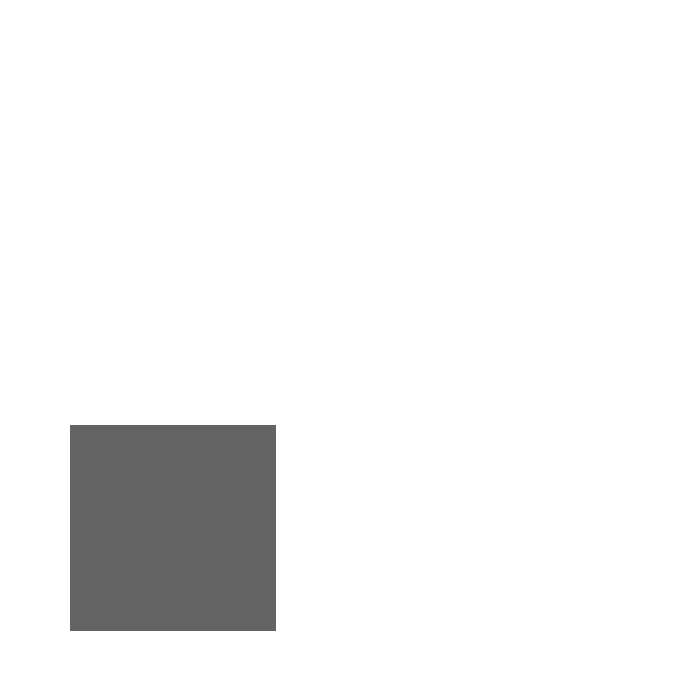}}
    \caption[short]{$A=\{y_1, y_4\}$}
\end{subfigure}%
\begin{subfigure}{.25\textwidth}
    \centering
    \frame{
    \includegraphics[width=0.25\textwidth]{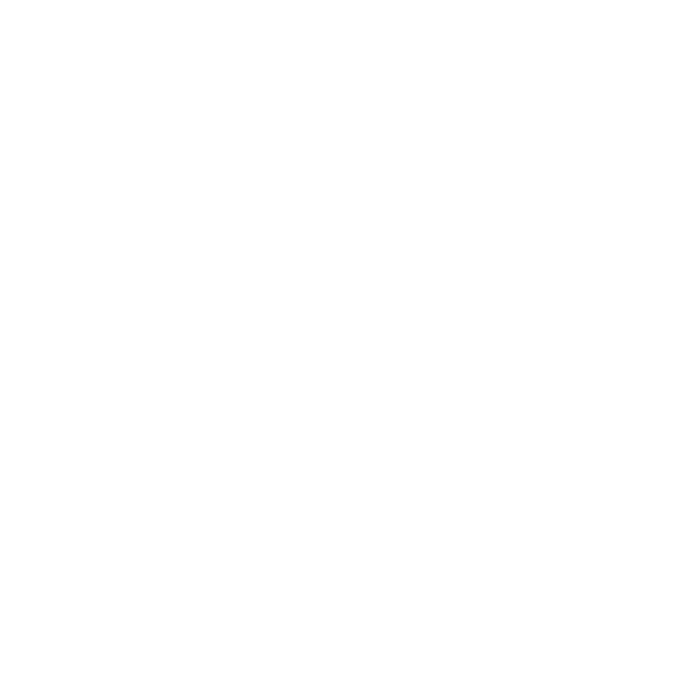}}
    \caption[short]{$A=\{y_2, y_3\}$}
\end{subfigure}
\begin{subfigure}{.25\textwidth}
    \centering
    \frame{
    \includegraphics[width=0.25\textwidth]{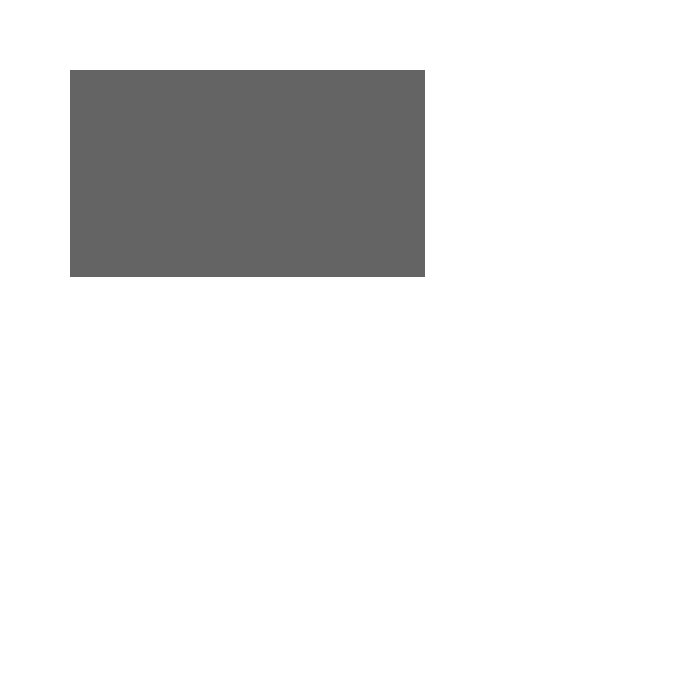}}
    \caption[short]{$A=\{y_2, y_4\}$}
\end{subfigure}%
\begin{subfigure}{.25\textwidth}
    \centering
    \frame{
    \includegraphics[width=0.25\textwidth]{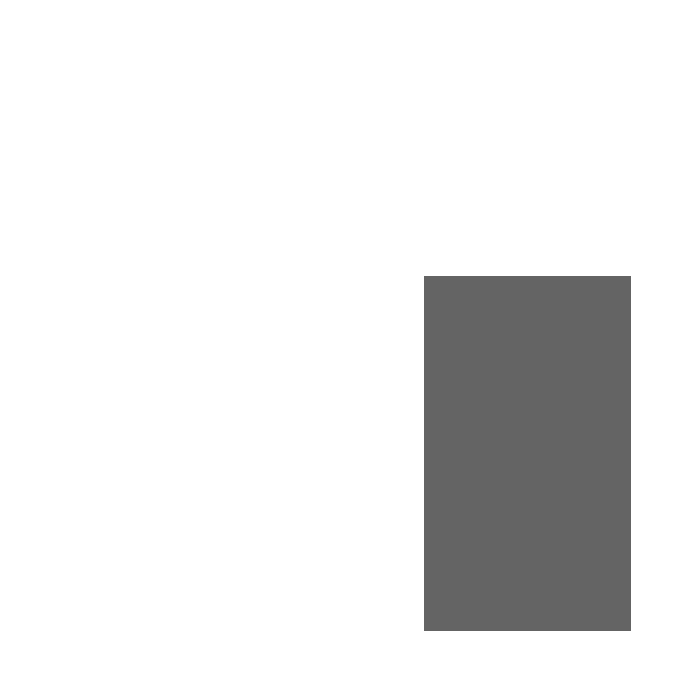}}
    \caption[short]{$A=\{y_3, y_4\}$}
\end{subfigure}%
\begin{subfigure}{.25\textwidth}
    \centering
    \frame{
    \includegraphics[width=0.25\textwidth]{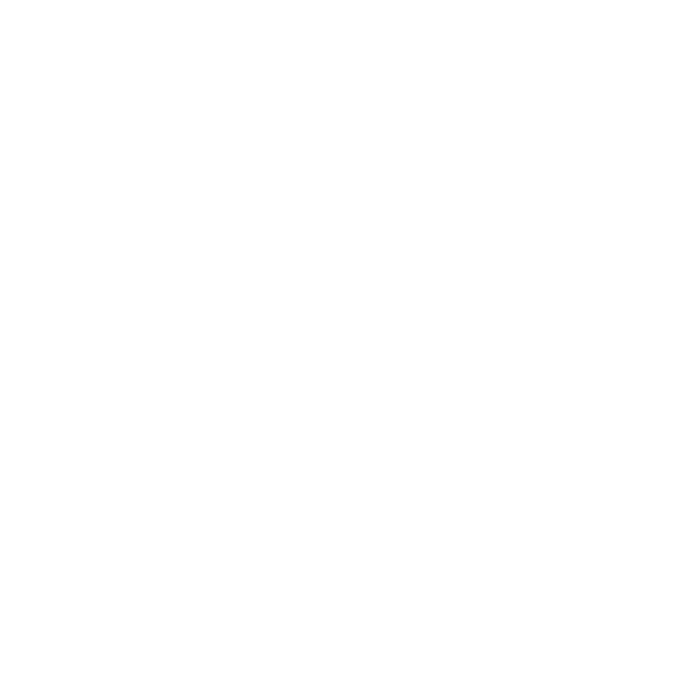}}
    \caption[short]{$A=\{y_1, y_2, y_3\}$}
\end{subfigure}%
\begin{subfigure}{.25\textwidth}
    \centering
    \frame{
    \includegraphics[width=0.25\textwidth]{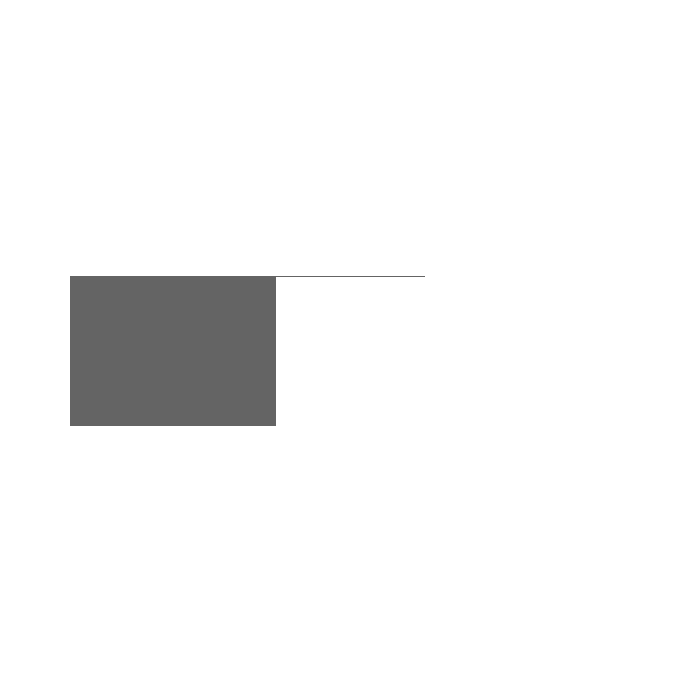}}
    \caption[short]{$A=\{y_1, y_2, y_4\}$}
\end{subfigure}
\raggedright
\begin{subfigure}{.25\textwidth}
    \centering
    \frame{
    \includegraphics[width=0.25\textwidth]{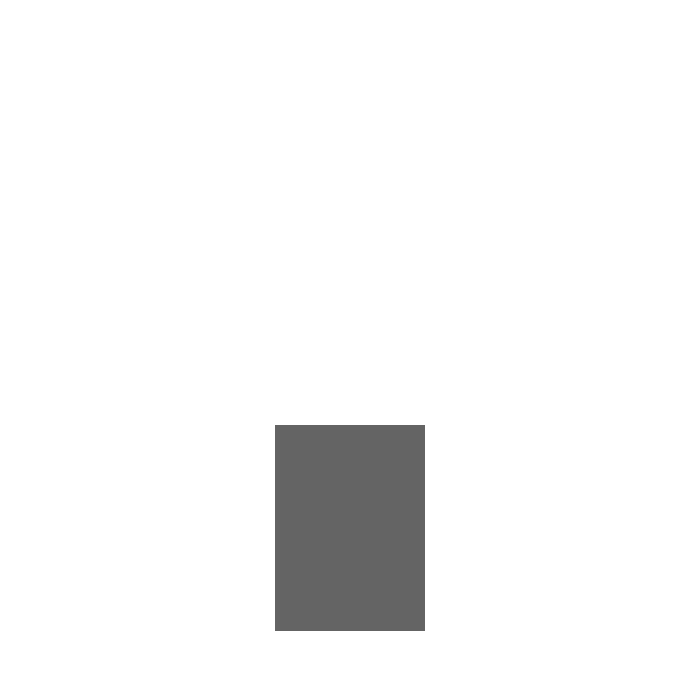}}
    \caption[short]{$A=\{y_1, y_3, y_4\}$}
\end{subfigure}%
\begin{subfigure}{.25\textwidth}
    \centering
    \frame{
    \includegraphics[width=0.25\textwidth]{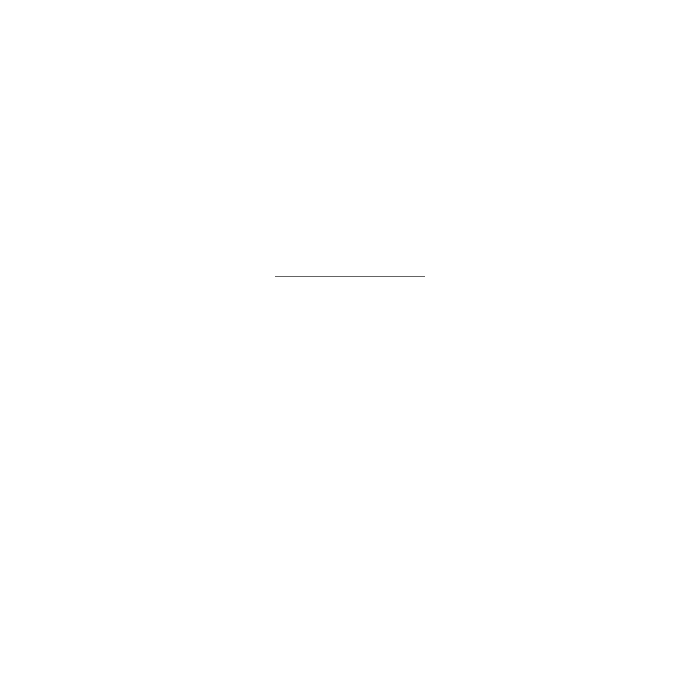}}
    \caption[short]{$A=\{y_2, y_3, y_4\}$}
\end{subfigure}%
\begin{subfigure}{.25\textwidth}
\centering
\frame{
    \includegraphics[width=0.25\textwidth]{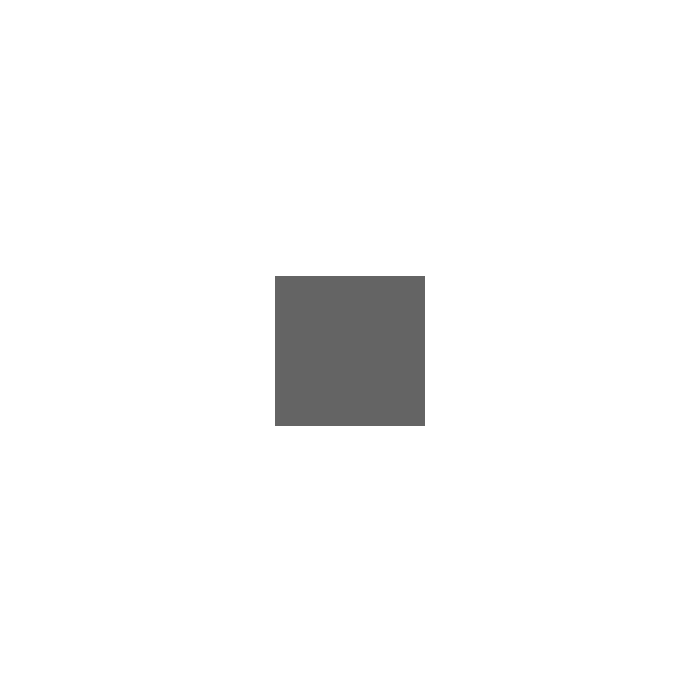}}
    \caption[short]{$A=\{y_1, y_2, y_3, y_4\}$}
\end{subfigure}%
\end{mdframed}
\caption{The cells $X_A$: Example \ref{ex2}}\label{figure2-2}
\end{figure}
\begin{ex}\label{ex3}
 $y_1=(0, 0)$, $y_2=(0, 4)$, $y_3=(4, 0)$, $y_4=(2, 2)$, $y_5=(1, 3)$, $y_6=(3, 3)$, $y_7=(3, 1)$.
 $$\nu=(0.15)\delta_{y_1}+(0.1)\delta_{y_2}+(0.1)\delta_{y_3}+(0.05)\delta_{y_4}+(0.2)\delta_{y_5}+(0.2)\delta_{y_6}+(0.2)\delta_{y_7}$$
\end{ex}

\begin{figure}[H]
\centering
\begin{mdframed}
\begin{minipage}{0.1\linewidth}
\begin{subfigure}{.05\textwidth}

\begin{tikzpicture}
\pgfplotsset{
colormap={revblackwhite}{gray(0cm)=(1); gray(1cm)=(0)}
}
\pgfplotscolorbardrawstandalone[
    colorbar,
    point meta min=0,
    point meta max=1,
    colorbar style={
        width=0.5cm}
        ]
\end{tikzpicture}%

\end{subfigure}%
\end{minipage}
\begin{minipage}{0.9\linewidth}
\begin{subfigure}{.32\textwidth}
    \centering
    \frame{
    \includegraphics[width=0.45\textwidth]{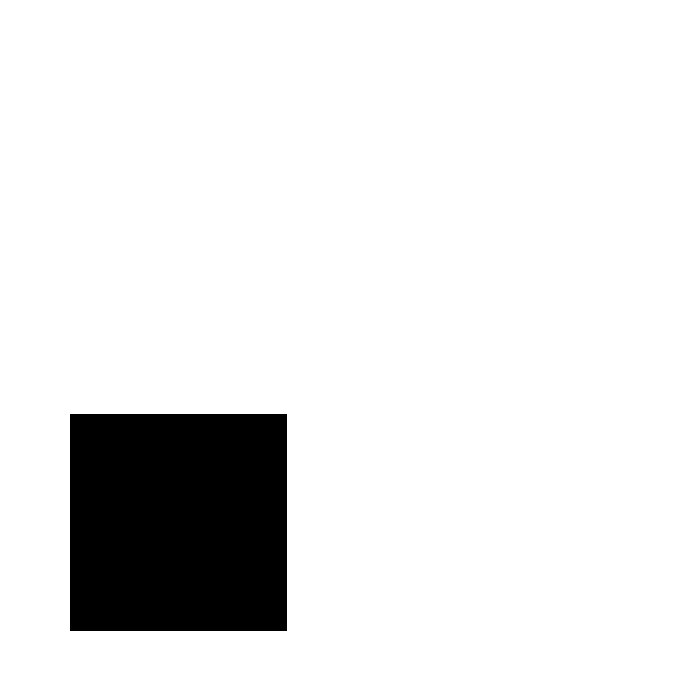}}
    \caption[short]{$\mu_1$}
\end{subfigure}%
\begin{subfigure}{.32\textwidth}
    \centering
    \frame{
    \includegraphics[width=0.45\textwidth]{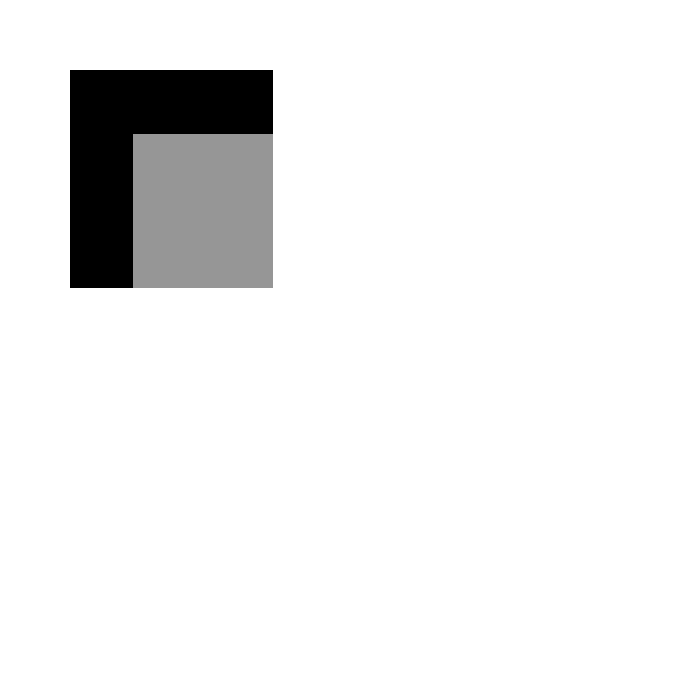}}
    \caption[short]{$\mu_2$}
\end{subfigure}%
\begin{subfigure}{.32\textwidth}
    \centering
    \frame{
    \includegraphics[width=0.45\textwidth]{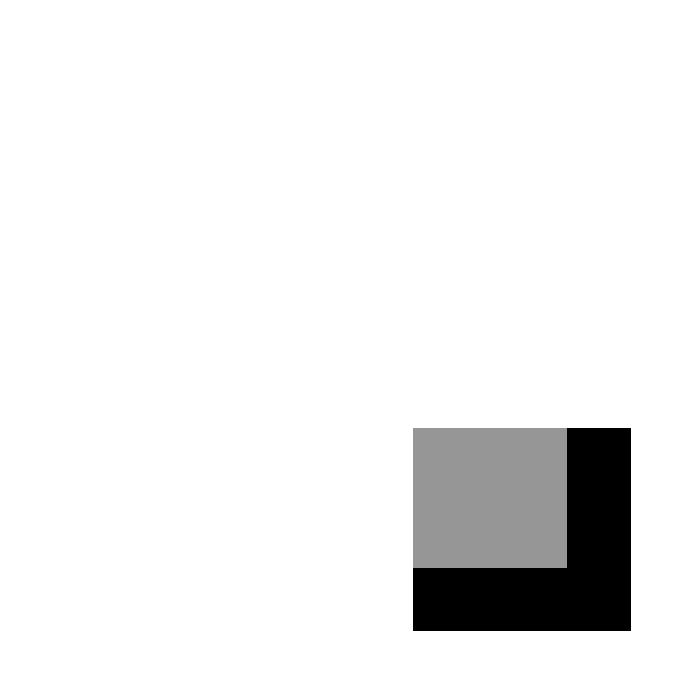}}
    \caption[short]{$\mu_3$}
\end{subfigure}
\begin{subfigure}{.32\textwidth}
    \centering
    \frame{
    \includegraphics[width=0.45\textwidth]{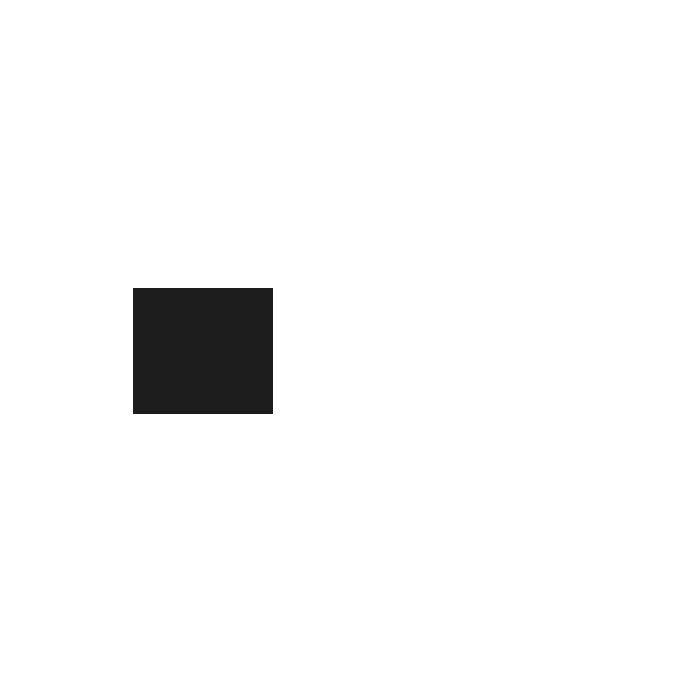}}
    \caption[short]{$\mu_4$}
\end{subfigure}%
\begin{subfigure}{.32\textwidth}
    \centering
    \frame{
    \includegraphics[width=0.45\textwidth]{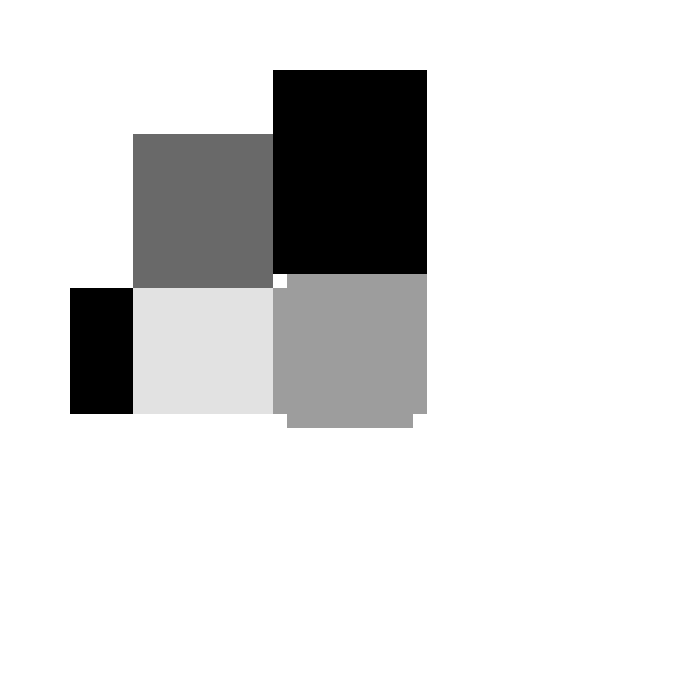}}
    \caption[short]{$\mu_5$}
\end{subfigure}%
\begin{subfigure}{.32\textwidth}
    \centering
    \frame{
    \includegraphics[width=0.45\textwidth]{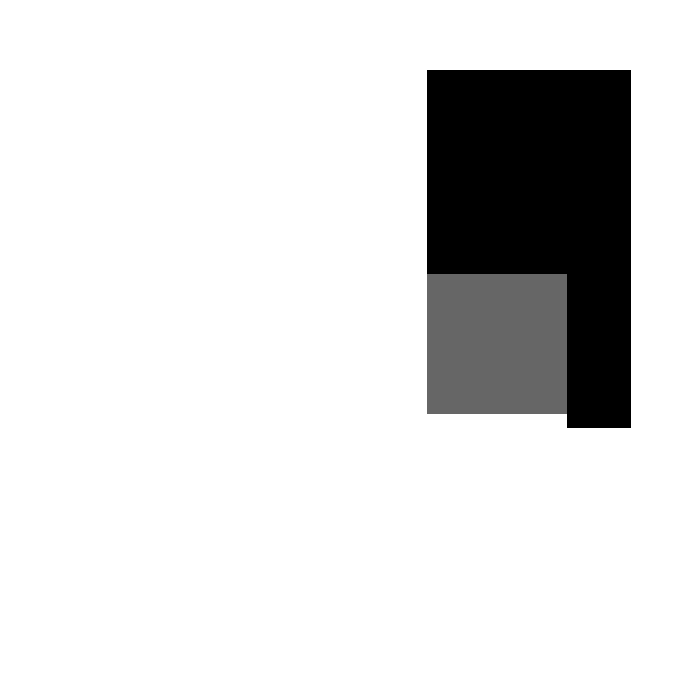}}
    \caption[short]{$\mu_6$}
\end{subfigure}
\begin{subfigure}{.32\textwidth}
    \centering
    \frame{
    \includegraphics[width=0.45\textwidth]{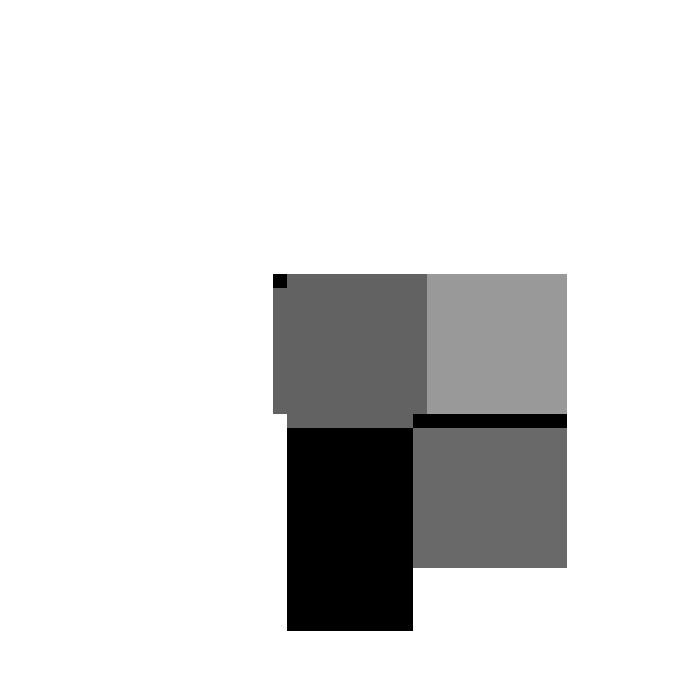}}
    \caption[short]{$\mu_7$}
\end{subfigure}
\end{minipage}
\end{mdframed}
\caption{Transportation of mass: Example \ref{ex3}}\label{figure3-1}
\end{figure}

\begin{figure}[H]
\begin{mdframed}
\centering
\begin{subfigure}{.2\textwidth}
    \centering
    \frame{
    \includegraphics[width=0.35\textwidth]{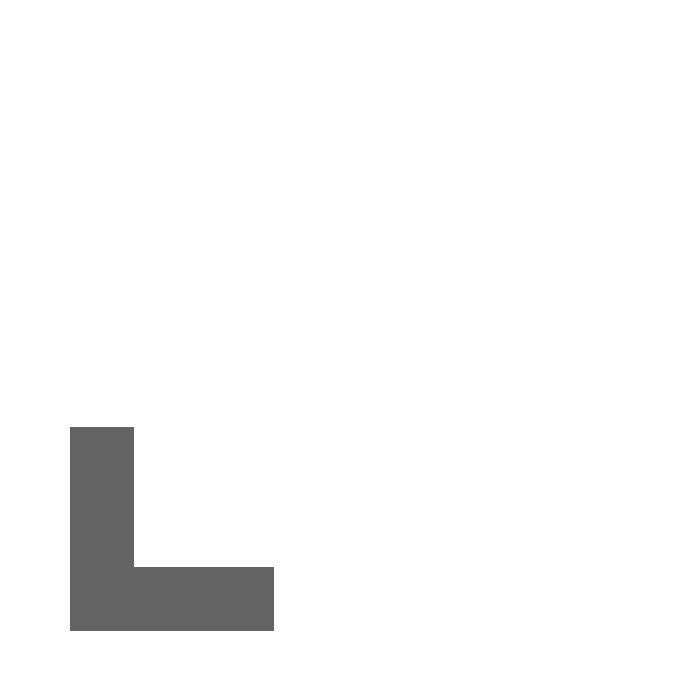}}
    \caption[short]{{\tiny $\{y_1\}$}}
\end{subfigure}%
\begin{subfigure}{.2\textwidth}
    \centering
    \frame{
    \includegraphics[width=0.35\textwidth]{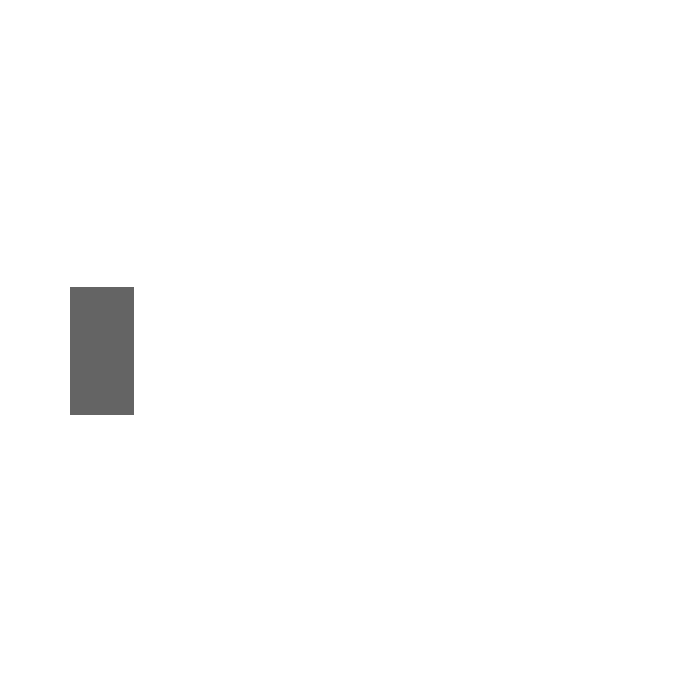}}
    \caption[short]{{\tiny $\{y_5\}$}}
\end{subfigure}%
\begin{subfigure}{.2\textwidth}
    \centering
    \frame{
    \includegraphics[width=0.35\textwidth]{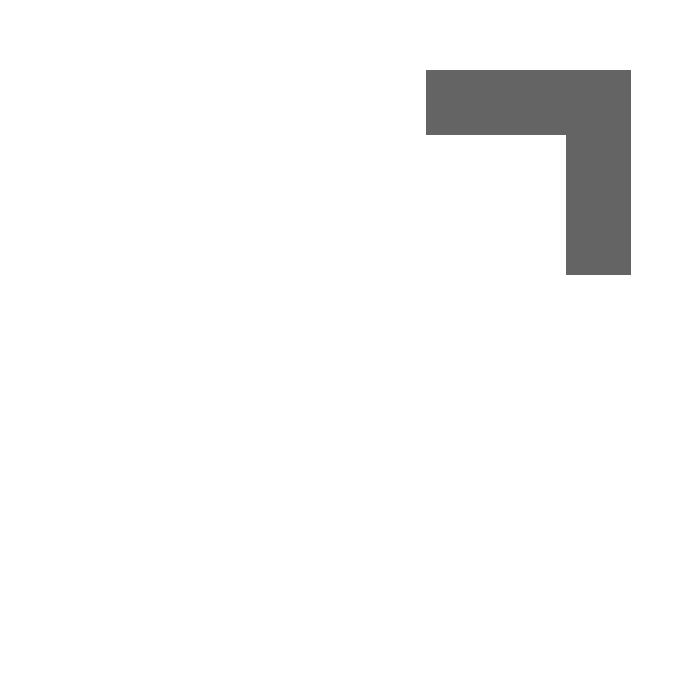}}
    \caption[short]{{\tiny $\{y_6\}$}}
\end{subfigure}%
\begin{subfigure}{.2\textwidth}
    \centering
    \frame{
    \includegraphics[width=0.35\textwidth]{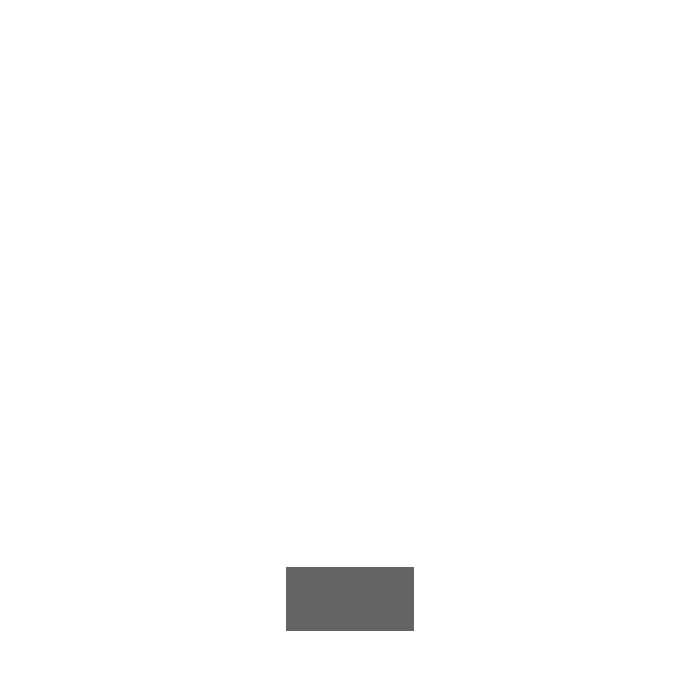}}
    \caption[short]{{\tiny $\{y_7\}$}}
\end{subfigure}%
\begin{subfigure}{.2\textwidth}
    \centering
    \frame{
    \includegraphics[width=0.35\textwidth]{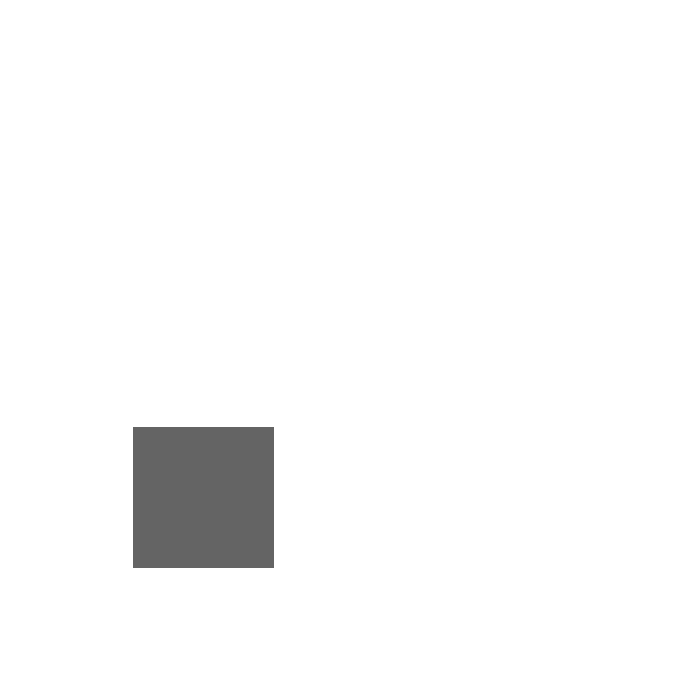}}
    \caption[short]{{\tiny $\{y_1, y_4\}$}}
\end{subfigure}
\begin{subfigure}{.2\textwidth}
    \centering
    \frame{
    \includegraphics[width=0.35\textwidth]{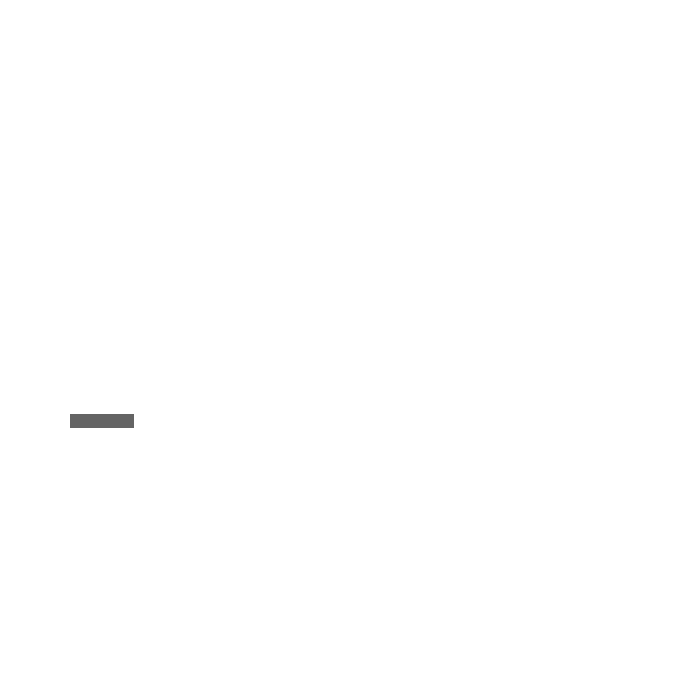}}
    \caption[short]{{\tiny $\{y_1, y_5\}$}}
\end{subfigure}%
\begin{subfigure}{.2\textwidth}
    \centering
    \frame{
    \includegraphics[width=0.35\textwidth]{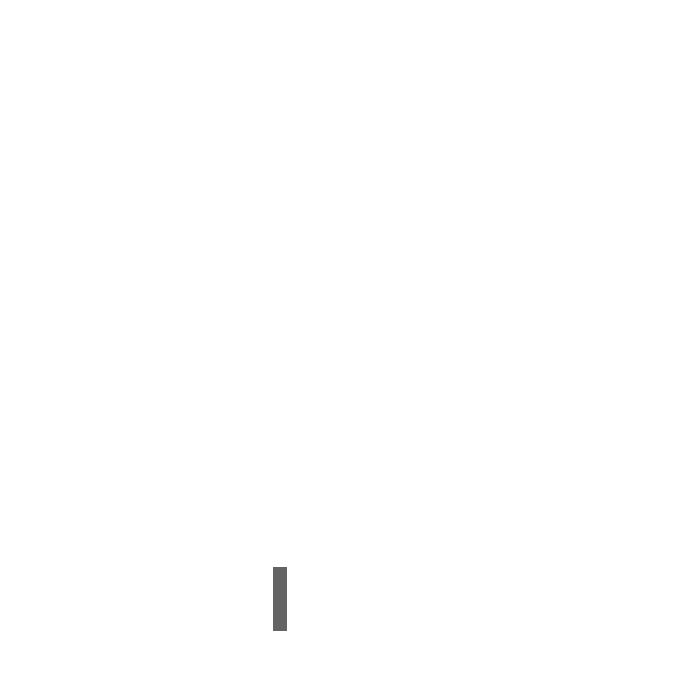}}
    \caption[short]{{\tiny $\{y_1, y_7\}$}}
\end{subfigure}%
\begin{subfigure}{.2\textwidth}
    \centering
    \frame{
    \includegraphics[width=0.35\textwidth]{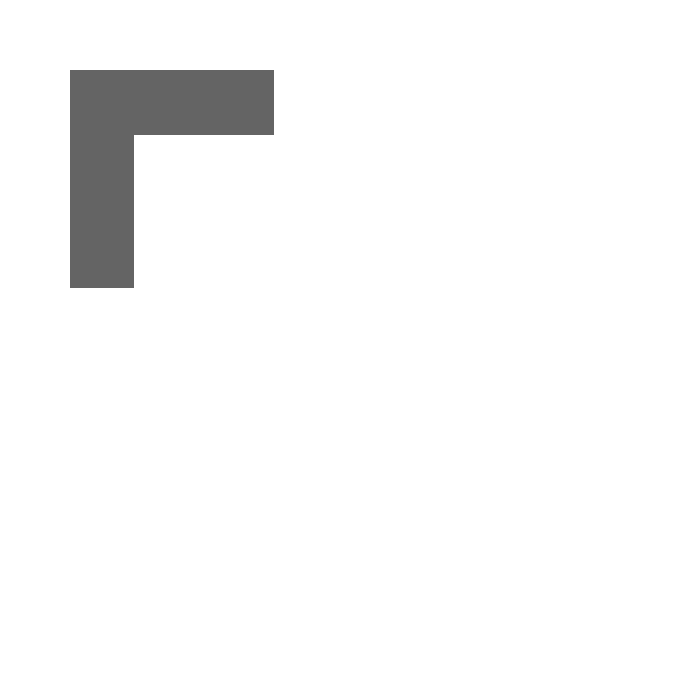}}
    \caption[short]{{\tiny $\{y_2, y_5\}$}}
\end{subfigure}%
\begin{subfigure}{.2\textwidth}
    \centering
    \frame{
    \includegraphics[width=0.35\textwidth]{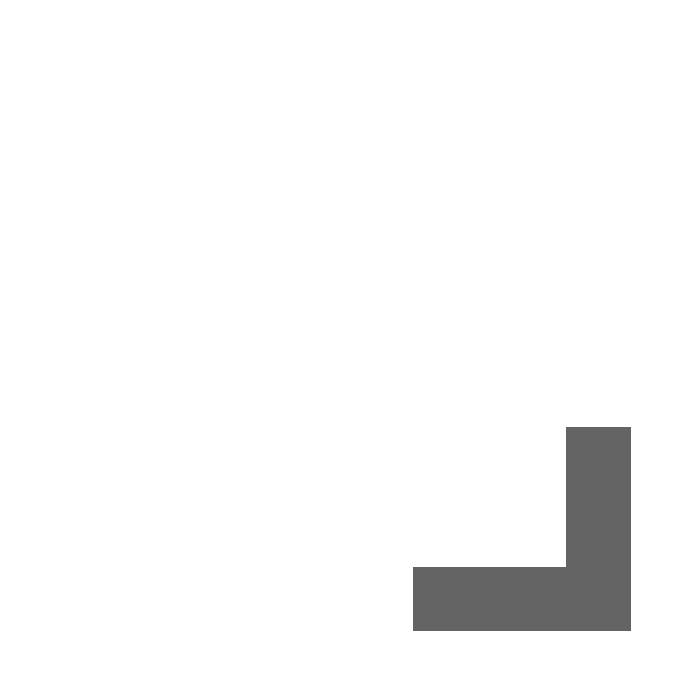}}
    \caption[short]{{\tiny $\{y_3, y_7\}$}}
\end{subfigure}%
\begin{subfigure}{.2\textwidth}
    \centering
    \frame{
    \includegraphics[width=0.35\textwidth]{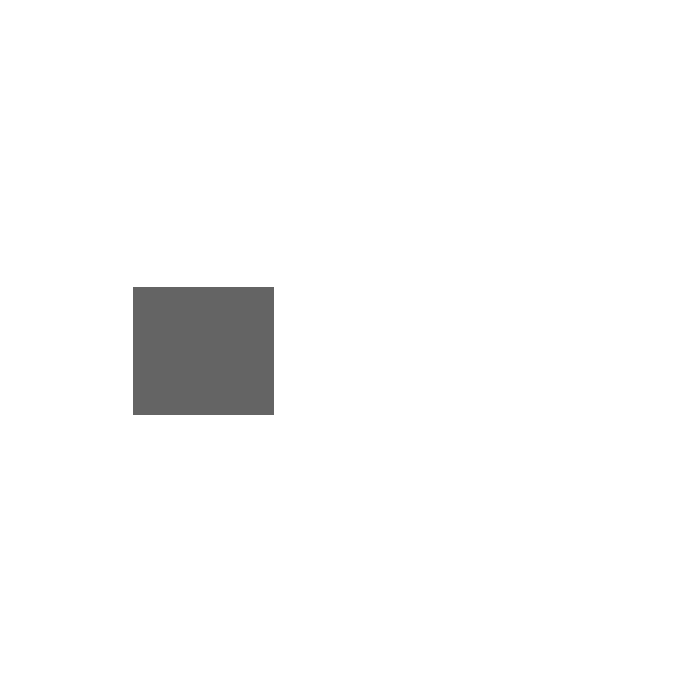}}
    \caption[short]{{\tiny $\{y_4, y_5\}$}}
\end{subfigure}
\begin{subfigure}{.2\textwidth}
    \centering
    \frame{
    \includegraphics[width=0.35\textwidth]{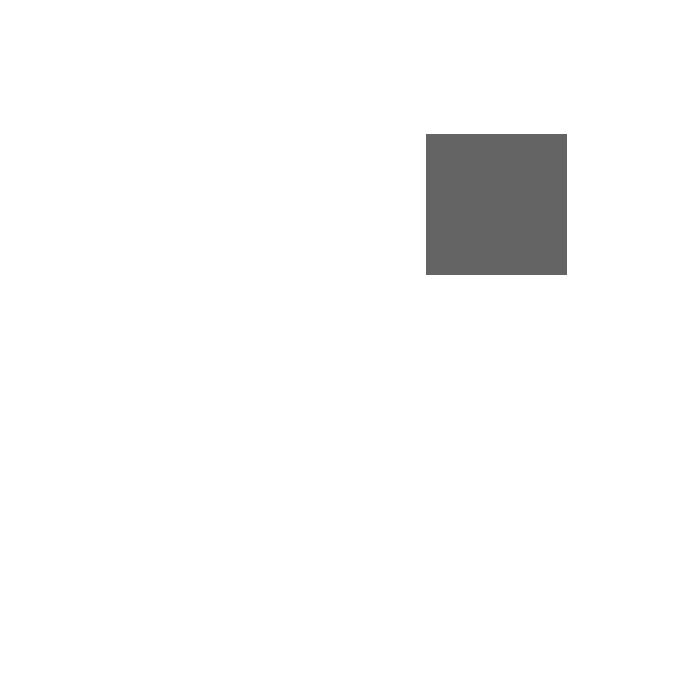}}
    \caption[short]{{\tiny $\{y_4, y_6\}$}}
\end{subfigure}%
\begin{subfigure}{.2\textwidth}
    \centering
    \frame{
    \includegraphics[width=0.35\textwidth]{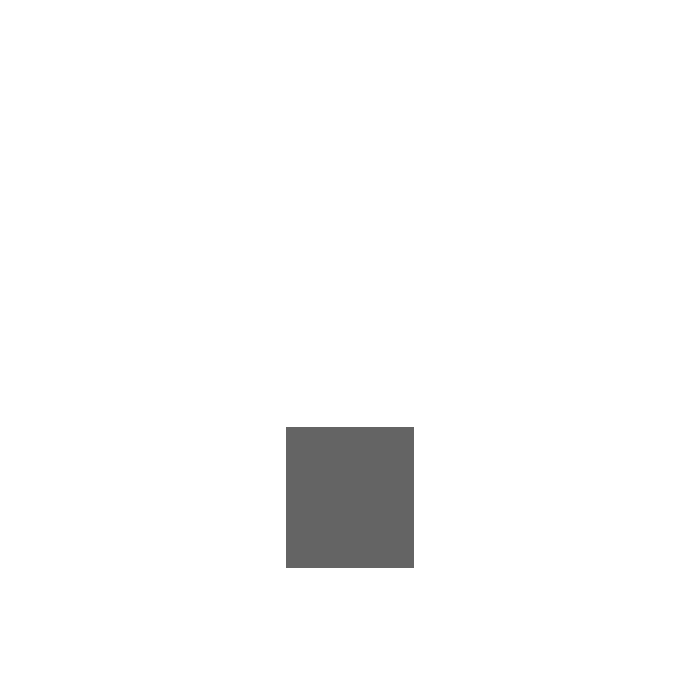}}
    \caption[short]{{\tiny $\{y_4, y_7\}$}}
\end{subfigure}%
\begin{subfigure}{.2\textwidth}
    \centering
    \frame{
    \includegraphics[width=0.35\textwidth]{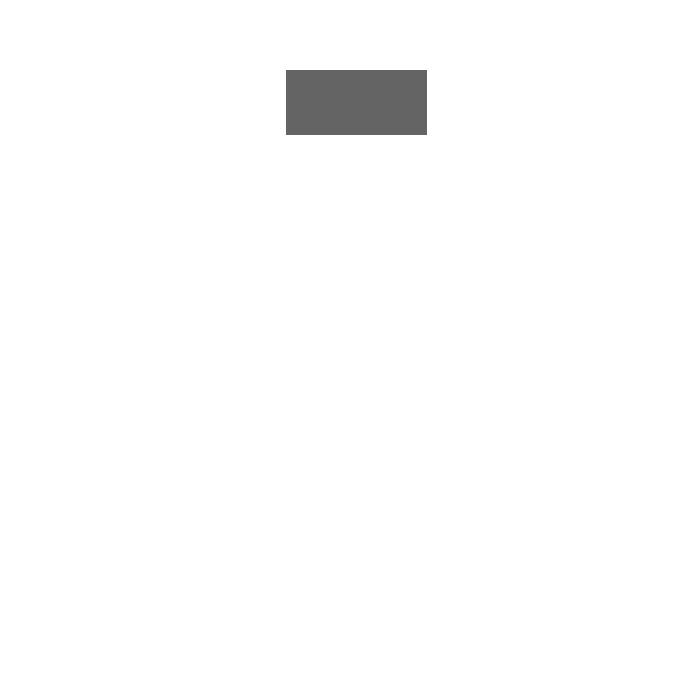}}
    \caption[short]{{\tiny $\{y_5, y_6\}$}}
\end{subfigure}%
\begin{subfigure}{.2\textwidth}
    \centering
    \frame{
    \includegraphics[width=0.35\textwidth]{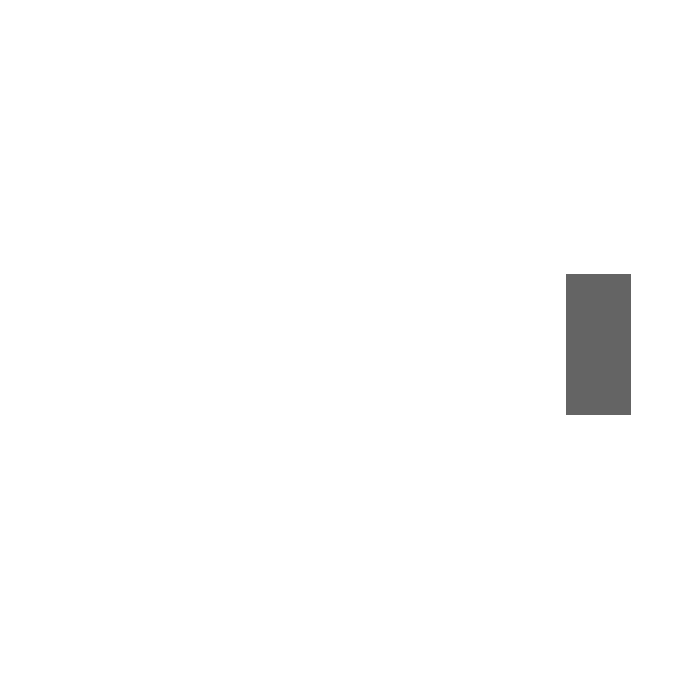}}
    \caption[short]{{\tiny $\{y_6, y_7\}$}}
\end{subfigure}%
\begin{subfigure}{.2\textwidth}
\centering
\frame{
    \includegraphics[width=0.35\textwidth]{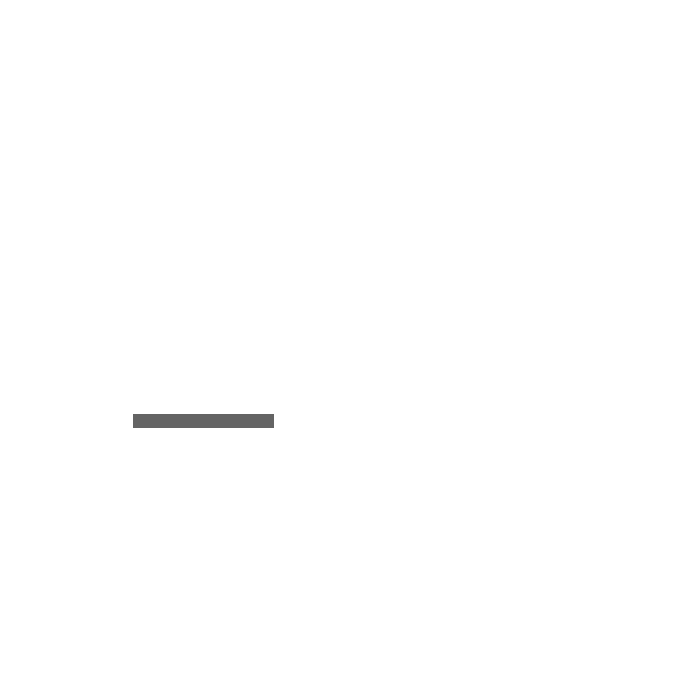}}
    \caption[short]{{\tiny $\{y_1, y_4, y_5\}$}}
\end{subfigure}
\begin{subfigure}{.2\textwidth}
\centering
\frame{
    \includegraphics[width=0.35\textwidth]{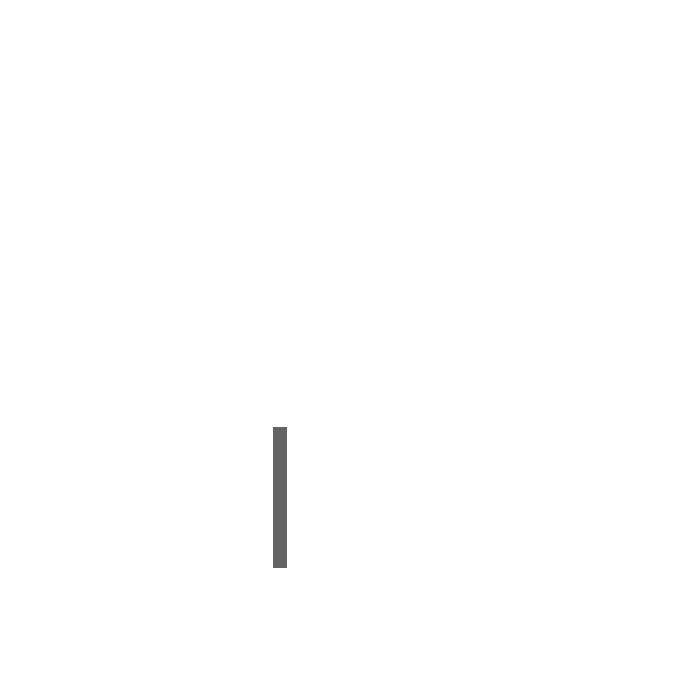}}
    \caption[short]{{\tiny $\{y_1, y_4, y_7\}$}}
\end{subfigure}%
\begin{subfigure}{.2\textwidth}
\centering
\frame{
    \includegraphics[width=0.35\textwidth]{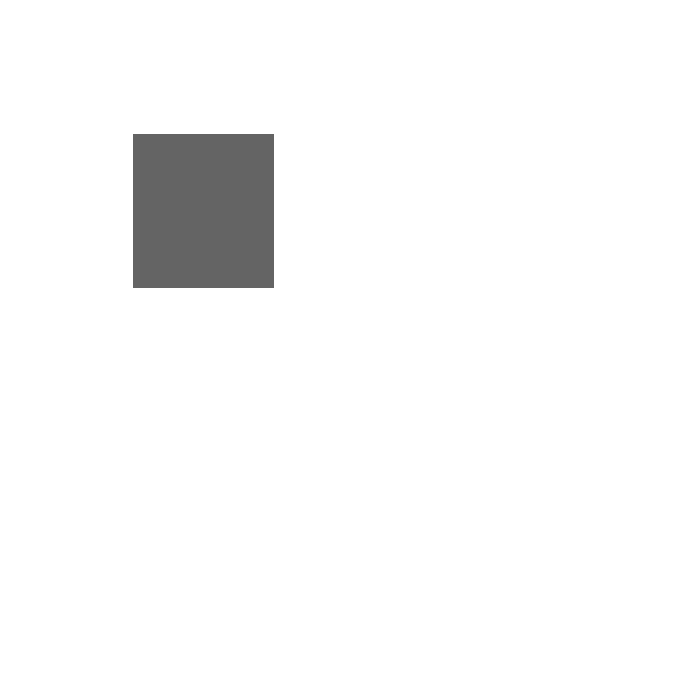}}
    \caption[short]{{\tiny $\{y_2, y_4, y_5\}$}}
\end{subfigure}%
\begin{subfigure}{.2\textwidth}
\centering
\frame{
    \includegraphics[width=0.35\textwidth]{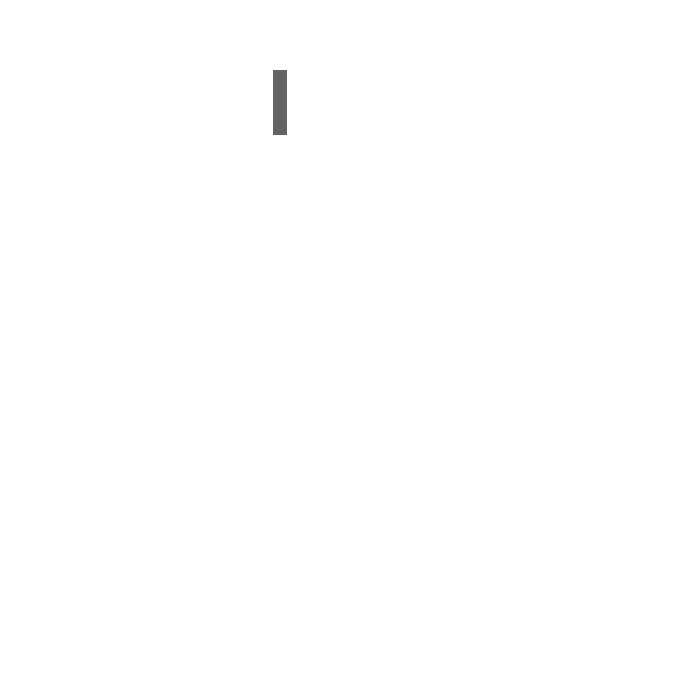}}
    \caption[short]{{\tiny $\{y_2, y_5, y_6\}$}}
\end{subfigure}%
\begin{subfigure}{.2\textwidth}
\centering
\frame{
    \includegraphics[width=0.35\textwidth]{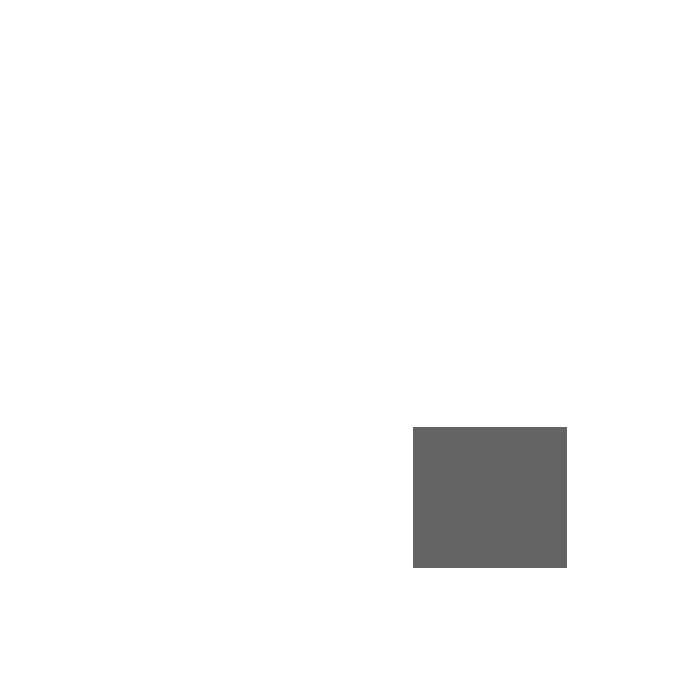}}
    \caption[short]{{\tiny $\{y_3, y_4, y_7\}$}}
\end{subfigure}%
\begin{subfigure}{.2\textwidth}
\centering
\frame{
    \includegraphics[width=0.35\textwidth]{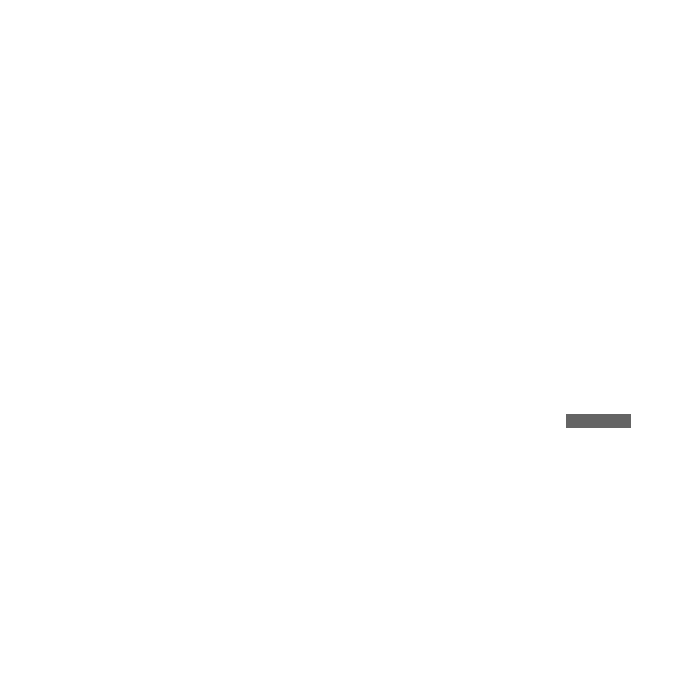}}
    \caption[short]{{\tiny $\{y_3, y_6, y_7\}$}}
\end{subfigure}
\begin{subfigure}{.2\textwidth}
\centering
\frame{
    \includegraphics[width=0.35\textwidth]{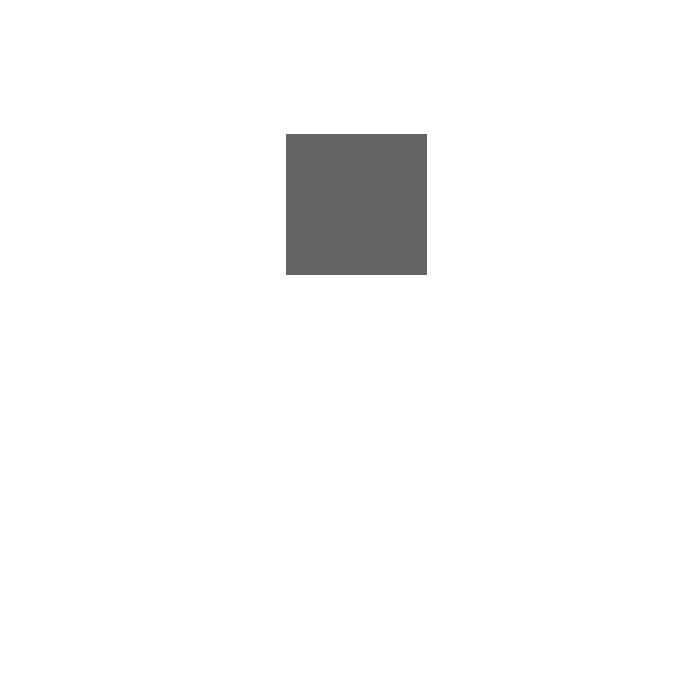}}
    \caption[short]{{\tiny $\{y_4, y_5, y_6\}$}}
\end{subfigure}%
\begin{subfigure}{.2\textwidth}
\centering
\frame{
    \includegraphics[width=0.35\textwidth]{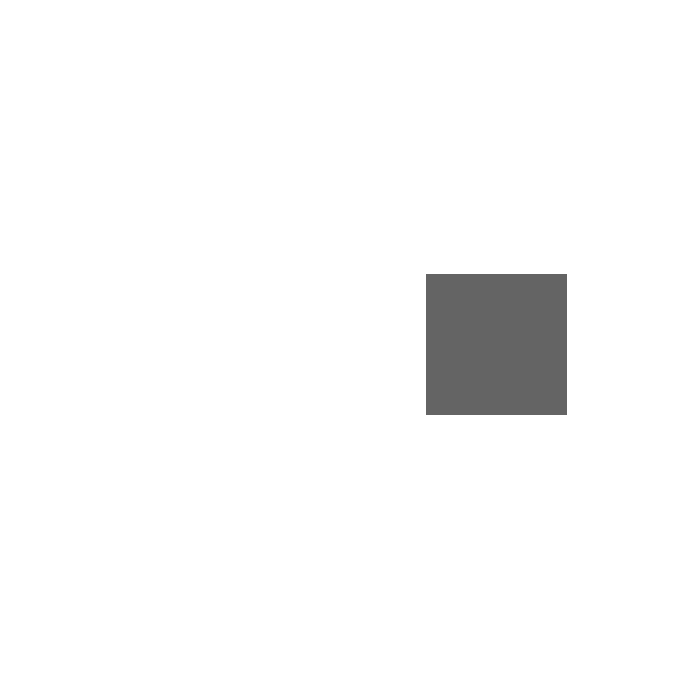}}
    \caption[short]{{\tiny $\{y_4, y_6, y_7\}$}}
\end{subfigure}%
\begin{subfigure}{.2\textwidth}
\centering
\frame{
    \includegraphics[width=0.35\textwidth]{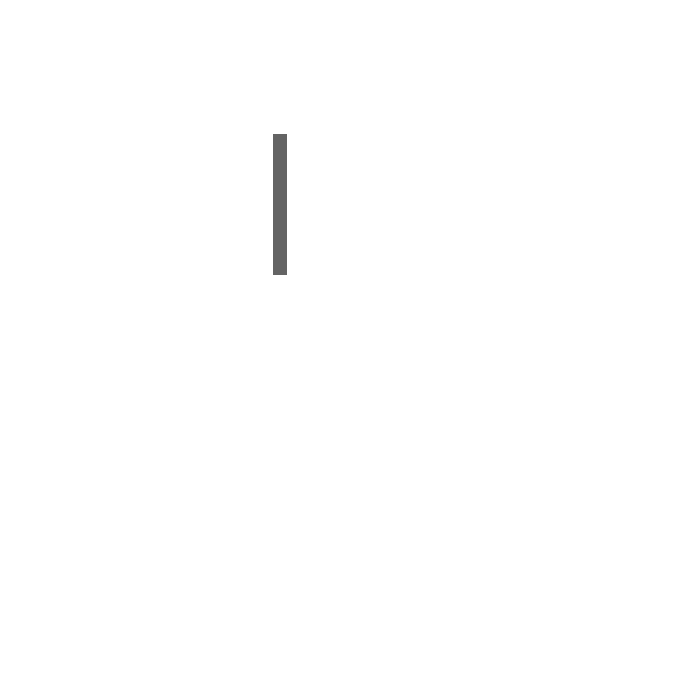}}
    \caption[short]{{\tiny $\{y_2, y_4, y_5, y_6\}$}}
\end{subfigure}%
\begin{subfigure}{.2\textwidth}
\centering
\frame{
    \includegraphics[width=0.35\textwidth]{Set_3Cell_Diagram2456.png}}
    \caption[short]{{\tiny $\{y_4, y_5, y_6, y_7\}$}}
\end{subfigure}%
\begin{subfigure}{.2\textwidth}
\centering
\frame{
    \includegraphics[width=0.35\textwidth]{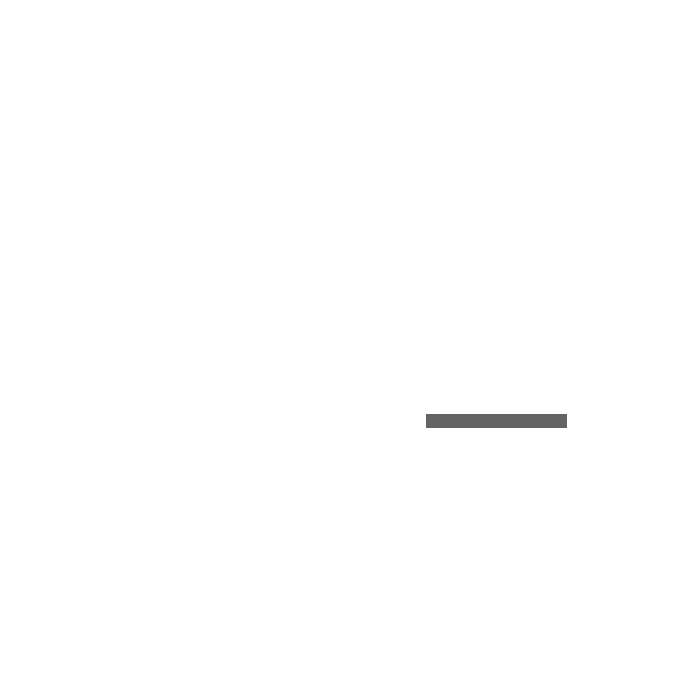}}
    \caption[short]{{\tiny $\{y_3, y_4, y_6, y_7\}$}}
\end{subfigure}
\raggedright\begin{subfigure}{.2\textwidth}
\centering
\frame{
    \includegraphics[width=0.35\textwidth]{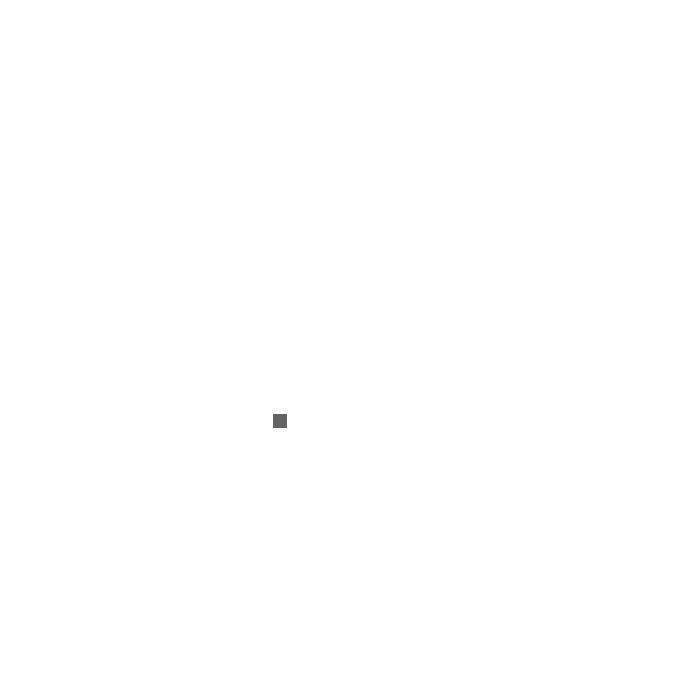}}
    \caption[short]{{\tiny $\{y_1, y_4, y_5, y_6, y_7\}$}}
\end{subfigure}%
\begin{subfigure}{.2\textwidth}
\centering
\frame{
    \includegraphics[width=0.35\textwidth]{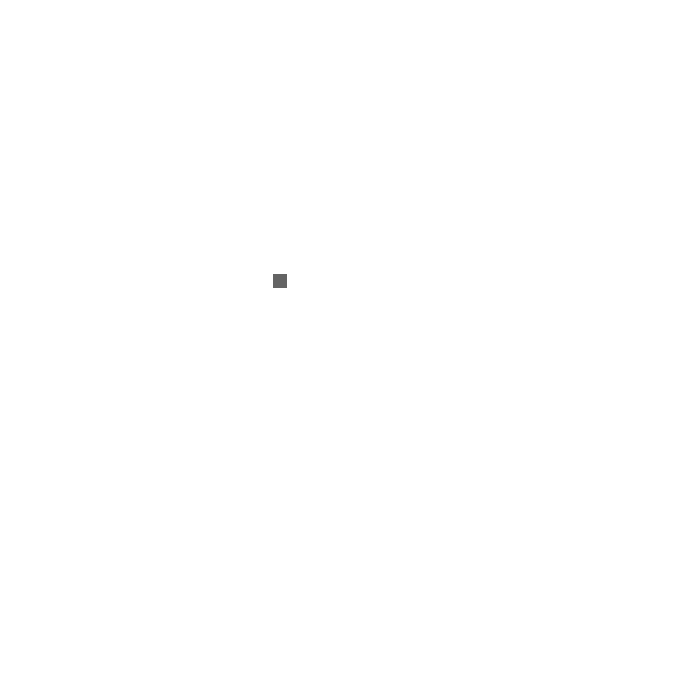}}
    \caption[short]{{\tiny $\{y_2, y_4, y_5, y_6, y_7\}$}}
\end{subfigure}%
\begin{subfigure}{.2\textwidth}
\centering
\frame{
    \includegraphics[width=0.35\textwidth]{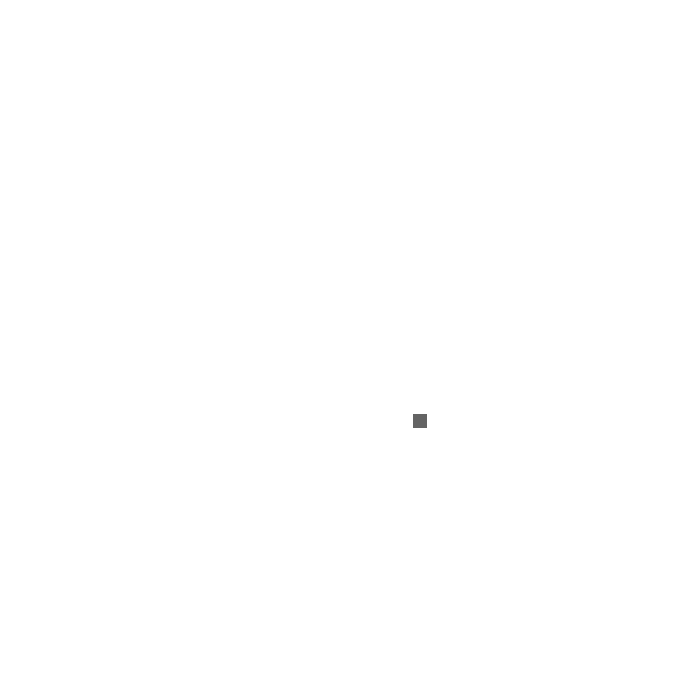}}
    \caption[short]{{\tiny $\{y_3, y_4, y_5, y_6, y_7\}$}}
\end{subfigure}
\end{mdframed}
\caption{The cells $X_A$: Example \ref{ex3}}\label{figure3-2}
\end{figure}

\bibliographystyle{alpha}
\bibliography{snowshovelingalg}

\end{document}